\documentclass[a4paper, reqno]{amsart}
\newif\iffinal
\finaltrue

\pdfoutput=1 

\usepackage[utf8]{inputenc}
\usepackage[T1]{fontenc}
\usepackage{lmodern}
\usepackage[english]{babel}
\usepackage[shortlabels]{enumitem}
\usepackage{amssymb}
\usepackage{colonequals}
\usepackage{mathtools}
\usepackage{bm}
\usepackage[pdfencoding=auto, colorlinks=true, linkcolor=blue, citecolor=blue]{hyperref}
\usepackage{colortbl}
\usepackage{color}
\usepackage{caption}
\usepackage{mathrsfs, verbatim, tikz, float}
\usetikzlibrary{shapes.geometric,fit,arrows.meta,decorations.markings,quotes}
\usepackage{letltxmacro}
\usepackage{todonotes}
\LetLtxMacro\todonotestodo\todo
\renewcommand{\todo}[2][]{\todonotestodo[backgroundcolor=yellow, #1]{TODO: {#2}}}
\newcommand{\attention}[2][]{\todonotestodo[backgroundcolor=red, #1]{TODO: {#2}}}
\iffinal\renewcommand{\todo}[2][]{}\renewcommand{\attention}[2][]{}\fi
\iffinal
\else
\usepackage[mark]{gitinfo2}
\renewcommand{\gitMark}{\jobname\,\textbullet{}\,\gitFirstTagDescribe\,\textbullet{}\,\gitAuthorName,\,\gitAuthorIsoDate}
\fi
\newcommand{\F}{\mathbb{F}}

\newcommand{\smat}[4]{\big[\begin{smallmatrix} #1 & #2 \\ #3 & #4 \end{smallmatrix}\big]}
\newcommand{\smatb}[4]{\Big[\begin{smallmatrix} #1 & #2 \\ #3 & #4 \end{smallmatrix}\Big]}
\newcommand{\longmid}{\,\,\middle\vert\,\,}
\newcommand{\mcB}{\mathcal{B}}
\newcommand{\mcC}{\mathcal{C}}
\newcommand{\mcD}{\mathcal{D}}
\newcommand{\mcL}{\mathcal{L}}

\newcommand{\mcU}{\mathcal{U}}

\newcommand{\msP}{\mathscr{P}}

\DeclareMathOperator{\cent}{cent}
\DeclareMathOperator{\GL2}{GL_2(\F_q)}
\DeclareMathOperator{\Int}{Int}

\DeclareMathOperator{\tr}{tr}

\newcommand{\mat}[1]{M_2(\F_{#1})}
\newcommand{\matFq}{\mat{q}}

\newcommand{\arxivid}{ (\url{https://arxiv.org/abs/2405.04106})}
\newcolumntype{C}[1]{>{\centering\let\newline\\\arraybackslash\hspace{0pt}}m{#1}}
\makeatletter
\newtheorem*{rep@theorem}{\rep@title}
\newcommand{\newreptheorem}[2]{%
\newenvironment{rep#1}[1]{%
 \def\rep@title{#2 \ref{##1}}%
 \begin{rep@theorem}}%
 {\end{rep@theorem}}}
\makeatother

\makeatletter
\@namedef{subjclassname@2020}{\textup{2020} Mathematics Subject Classification}
\makeatother

\AtBeginDocument{%
   \def\MR#1{}
}

\newreptheorem{theorem}{Theorem}

\newtheorem{theorem}{Theorem}
\newtheorem{proposition}{Proposition}[section]
\newtheorem{corollary}[proposition]{Corollary}
\newtheorem{lemma}[proposition]{Lemma}
\theoremstyle{definition}
\newtheorem{definition}[proposition]{Definition}
\newtheorem{example}[proposition]{Example}

\newtheorem{fact}[proposition]{Fact}
\theoremstyle{remark}
\newtheorem{remark}[proposition]{Remark}
\newtheorem*{rem*}{Remark}
\numberwithin{equation}{section}

\author{Roswitha Rissner}
\address{Department of Mathematics\\University of Klagenfurt\\
  Universitätsstraße 65-67\\9020 Klagenfurt am Wörthersee\\Austria}
\email{\href{mailto:roswitha.rissner@aau.at}{roswitha.rissner@aau.at}}
\thanks{This research was funded in part by the Austrian Science
  Fund~(FWF)~[10.55776/DOC78]. For open access purposes, the authors
  have applied a CC~BY public copyright license to any author-accepted
  manuscript version arising from this submission.}

\author{Nicholas J.~Werner}
\address{Department of Mathematics\\ Computer and Information Science \\
  SUNY at Old Westbury\\  Old Westbury\\NY 11568\\USA}
\email{\href{mailto:wernern@oldwestbury.edu}{wernern@oldwestbury.edu}}

\title[%
  \textit{C\MakeLowercase{ounting core sets}}
]{%
  Counting core sets in matrix rings over finite fields}

\begin{document}

\begin{abstract}
Let $R$ be a commutative ring and $M_n(R)$ be the ring of $n \times n$
matrices with entries from $R$. For each $S \subseteq M_n(R)$, we
consider its (generalized) null ideal $N(S)$, which is the set of all
polynomials $f$ with coefficients from $M_n(R)$ with the property that
$f(A) = 0$ for all $A \in S$. The set $S$ is said to be core if $N(S)$
is a two-sided ideal of $M_n(R)[x]$. It is not known how common core
sets are among all subsets of $M_n(R)$. We study this problem for
$2 \times 2$ matrices over $\F_q$, where $\F_q$ is the finite field
with $q$ elements. We provide exact counts for the number of core
subsets of each similarity class of $M_2(\F_q)$. While not every
subset of $M_2(\F_q)$ is core, we prove that as $q \to \infty$, the
probability that a subset of $M_2(\F_q)$ is core approaches 1. Thus,
asymptotically in~$q$, almost all subsets of $M_2(\F_q)$ are core.


\end{abstract}

\maketitle

\section{Introduction}
\label{section:introduction}
It is well known that an element $a$ of a commutative ring $R$ is a
root of a polynomial $f\in R[x]$ if and only if $x-a$ divides
$f$. However, this need not hold when $R$ is noncommutative. Given
polynomials $f, g \in R[x]$, let $fg$ denote their product in
$R[x]$. If $R$ is noncommutative, then evaluation at a ring element
$a \in R$ need not define a multiplicative map $R[x] \to R$; that is,
it may happen that $(fg)(a) \neq f(a)g(a)$. For example, let $K$ be a
field and let $M_n(K)$ be the ring of $n \times n$ matrices with
entries from $K$. Take $A$ and $B$ to be two non-commuting matrices in
$M_n(K)$ and set $f(x)=x-A$ and $g(x)=x-B$. Then,
$(fg)(x) = x^2-(A+B)x+AB$, and $(fg)(A) = -BA + AB \neq 0$, but
$f(A)g(A) = 0$. Observe that this is an example for a polynomial that
is divisible by $x-A$ but does not vanish on $A$. The characterization
of the polynomials that vanish on an element or a set of elements of a
noncommutative ring, therefore, poses a challenge.

We study the set of polynomials with matrix coefficients that vanish
on a set of square matrices. Let $R$ be a commutative ring. Given
$S \subseteq M_n(R)$, the \emph{null ideal} of $S \subseteq M_n(R)$ is
defined as
\begin{equation}\label{eq:general-null-ideal}
 N(S)=  \{f\in M_n(R)[x] \mid f(A) = 0 \text{ for all } A \in S\}.
\end{equation}
Throughout, we assume that the variable $x$ commutes with all elements
in $M_n(R)$ and that polynomials are evaluated from the right, as is a
common convention in the literature
(cf.~\cite[§16]{Lam:2001:non-comm-rings}). It is easily verified that
$N(S)$ is a left ideal of
$M_n(R)[x]$~(\cite[Proposition~3.1(1)]{Werner:2022:null-ideals}).  It
may, however, fail to be a right ideal as the example above
demonstrates for $N(\{A\})$.

\begin{definition}
  We call a set $S \subseteq M_n(R)$ \emph{core} if its null ideal is
  a two-sided ideal of $M_n(R)$, and \emph{noncore} otherwise.
\end{definition}

\begin{remark}
  The definition of a core set in this paper should be not confused
  with the use of the term ``core'' in graph theory, where it refers
  to a graph $X$ in which every endomorphism of $X$ is an automorphism
  \cite[Section 6.2]{Godsil-Royle:2001}. In particular, our notion of
  core and the results in this article are unrelated to the work done
  in recent papers such as \cite{Orel:2012,Orel:2016}, in which
  certain graphs constructed using matrices over finite fields are
  shown to be cores in the graph theoretical sense.
\end{remark}

The purpose of this paper is to study null ideals and core sets in
$\matFq$, where $\F_q$ denotes the finite field with $q$
elements. Originally, null ideals of matrices are a notion in
classical linear algebra, where they describe the set of all
polynomials with coefficients from a field $K$ that vanish on a matrix
in $M_n(K)$.  Transferring the concept to the setting of a commutative
ring $R$ instead of a field $K$ already adds complexity as the
polynomial ring $R[x]$ is in general not a principal ideal domain
anymore. Determining the generators of $\{ f\in R[x] \mid f(S) = 0 \}$
for a subset $S\subseteq M_n(R)$ is the subject of research in the
past decades, see e.g.~\cite{Brown:1998:null-ideals-spanning-rank,
  Brown:1999:null-ideals-spanning-rank, Brown:2005:null-ideals,
  Heuberger-Rissner:2017:null-ideals,
  Hyun-Neiger-Schost:2021:lin-rec-seq, Rissner:2016:null-ideals}.

The study of null ideals is also strongly motivated by their
connection to integer-valued polynomials. Originally a notion
associated to commutative ring theory
(see~\cite{Cahen-Chabert:1997:ivp, Cahen-Chabert:2016:ivp-survey} for
a thorough treatment on the topic), mathematicians have begun to study
integer-valued polynomials over noncommutative rings in recent
years~\cite{Evrard-Fares-Johnson:2013:ivp-lower-triangular-matrices,
  Evrard-Johnson:2015:ivp-2by2, Frisch:2017:ivp-upper-triangular,
  Naghipour-Rismanchian-Hafshejani:2017:ivp-matrix-rings,
  Hafshejani-Naghipour-Rismanchian:2020:ivp-block-matrices,
  Hafshejani-Naghipour-Sakzad:2019:ivp-matrix-rings,
  Werner:2012:ivp-matrix-rings, Werner:2017:ivp-survey}.  The
connection to null ideals in the sense of this paper is the following:
for a subset $S\subseteq M_n(R)$, we want to study
\begin{equation*}
  \Int(S, M_n(R)) := \{f\in M_n(K)[x] \mid f(S) \subseteq M_n(R)\},
\end{equation*}
where $R$ is a commutative integral domain and $K$ its quotient
field. As the evaluation map $M_n(K)[x] \to M_n(K)$ is not a ring
homomorphism anymore, it may happen that $\Int(S, M_n(R))$ is not a
ring (w.r.t.~the usual addition and multiplication of polynomials). It
is, however, known that $\Int(S, M_n(R))$ is a ring if and only if the
image of $S$ in $M_n(R/aR)$ is a core set for all nonzero $a \in R$;
see~\cite[Section~2]{Werner:2022:null-ideals} for details. Hence, one
way to characterize subsets $S$ for which $\Int(S, M_n(R))$ is a ring
is to characterize the core subsets of residue rings of $M_n(R)$.

Recently, the second author studied null ideals in the general setting
as defined above in~\eqref{eq:general-null-ideal} with a particular
focus on $2\times 2$ and $3\times 3$ matrices over a field
\cite{Swartz-Werner:2023:null-ideal-3by3, Werner:2022:null-ideals}. It
is easy to see that any matrix ring $M_n(R)$ with $n \geq 2$ will
contain both core subsets and noncore subsets. As mentioned above, a
singleton set $\{A\} \subseteq M_n(R)$ may fail to be core. On the
other hand, both the empty set and $M_n(R)$ itself are core. Less
trivially, if $S \subseteq M_n(R)$ consists of scalar matrices, then
evaluation of polynomials at elements of $S$ behaves as in the
ordinary commutative setting, and $S$ is always core in this
case. More generally, the full similarity class of a matrix $A$ is
always core \cite[Proposition 3.1]{Werner:2022:null-ideals}. The
article \cite{Werner:2022:null-ideals} includes a full algorithmic
characterization of the core subsets of $\matFq$, and some partial
results on core subsets in $M_n(\F_q)$ with $n \geq 3$ are given in
\cite{Swartz-Werner:2023:null-ideal-3by3, Werner:2022:null-ideals}.

A natural question which arises is: how common are core sets? Working
in $M_n(K)$, where $K$ is a field, we would like to determine (or at
least bound) the probability that a randomly selected subset of
$M_n(K)$ is core. In this paper, we accomplish this goal for
$\matFq$. We determine exact counts for the number of core subsets in
each similarity class of $\matFq$, and using this we show that as
$q \to \infty$, the probability that a randomly chosen subset of
$\matFq$ is core tends to $1$. Thus, asymptotically in $q$, almost all
subsets of $\matFq$ are core.

A concise summary of our main results follows.


\subsection{Results}
\label{section:results}
In this paper, we present counts for the core subsets of $\matFq$
where $\F_q$ denotes the finite field of cardinality $q$. We identify
$\F_q$ with the field of scalar matrices in $\matFq$, so that
$\F_q \subseteq \matFq$. Following~\cite{Werner:2022:null-ideals} our
approach is to first enumerate all core subsets of each minimal
polynomial class of $\matFq$. For a polynomial $m \in \F_q[x]$, we set
$\mcC(m)$ to be the set of matrices in $\matFq$ whose minimal
polynomial is $m$, and we call this the minimal polynomial class of
$m$. For $2 \times 2$ matrices, minimal polynomial classes coincide
with similarity classes. That is, $A$, $B \in \mcC(m)$ if and only if
$A$ and $B$ are $\GL2$-conjugates. This can be proved by using the
rational canonical forms of matrices (see, for instance,
\cite[Section~12.2, Exercise~2]{DummitFoote}). Consequently, we will
use the terms ``minimal polynomial class'' and ``similarity class''
interchangeably.

The minimal polynomials of matrices in $\matFq$ can be categorized
based on their factorization in $\F_q[x]$ into irreducible
polynomials. We distinguish four types of polynomials: linear
polynomials, quadratic irreducible polynomials, split quadratic
polynomials with a repeated root (i.e.,\ those of the form $(x-a)^2$),
and split quadratic polynomials with distinct roots (those of the form
$(x-a)(x-b)$ with $a \ne b$). We call these types LIN, IRR, SQR, and
SQD, respectively. Given a minimal polynomial class $\mcC$ in
$\matFq$, we will say that $\mcC$ is LIN, IRR, SQR, or SQD, depending
on the type of the polynomial corresponding to $\mcC$.

In Theorems~\ref{theorem:IRR},~\ref{SQR thm}, and~\ref{theorem:SQD} we
give exact counts for the number of (non)core subsets within each of
the four types of minimal polynomial classes.  Moreover, in
Proposition~\ref{Class sizes prop} we provide the size of each class,
and the number of such classes in $\matFq$. See Table~\ref{Noncore
  table} for a summary of these counts. Note that we consider the
empty set to be core.

For a quadratic minimal polynomial class $\mcC$, it is known
\cite[Theorem 5.14]{Werner:2022:null-ideals} that a nonempty subset
$S \subseteq \mcC$ is core if and only if there exist $A$, $B \in S$
such that $A-B$ is invertible. Thus, for IRR, SQR, and SQD classes,
the number of noncore subsets equals the number of subsets $S$ with
the property that $A-B$ is singular for all $A$, $B\in S$
(cf.~Corollaries~\ref{SQR cor} and~\ref{SQD cor}). We note that this
condition can be used to define a \textit{matrix graph} (also called a
\textit{bilinear forms graph}) as
in~\cite{Huang-Huang-Li-Sze:2014}. For this construction, one forms a
graph whose vertex set is $M_2(\F_q)$, and two vertices $A$ and $B$
are adjacent if and only if $A-B$ has rank one. A noncore set
$S \subseteq \mcC$ forms a clique in such a matrix graph. These
objects have also been studied---sometimes under different
terminology---in recent articles such as~\cite{Huang-Semrl:2016}
and~\cite{Chooi-Kwa-Lim:2017}.

\begin{table}[h]
\centering
  \begin{tabular}{|c|c|c|c|}
  \hline
   class type &  \# classes  & size of class & \# noncore subsets in class \\
   \hline
   \hline
    LIN        &   $q$        & $1$           &   0                        \\
    \hline
    IRR        &$\binom{q}{2}$& $q^2-q$       &  $q^2-q$                   \\
    \hline
    SQR        &   $q$        & $q^2-1$       &  $(q+1)(2^{q-1} -1)$        \\
    \hline
    SQD        &$\binom{q}{2}$& $q^2+q$       &  $(q+1)(2^{q+1} -q-2)$      \\
    \hline
  \end{tabular}
  \caption{Counts of noncore subsets in minimal polynomial
    classes}\label{Noncore table}
\end{table}

Subsets $S$ which are not contained in a single minimal polynomial
class can be partitioned into the subsets $S_i = S\cap \mcC(m_i)$, one
for each minimal polynomial $m_i$ that occurs among the matrices in
$S$. In that respect, the following results
from~\cite{Werner:2022:null-ideals} play a crucial role:
\begin{enumerate}
\item Unions of core sets are again core
  (\cite[Proposition~3.1]{Werner:2022:null-ideals}).
\item Subsets $S_i$ of type LIN are core (as they contain at most one
  element, which is in the center of $\matFq$).
\item If $S$ is core, then each $S_i$ of type IRR or SQR is core
  (\cite[Corollary~5.4, Proposition~5.12]{Werner:2022:null-ideals}).
\end{enumerate}
In light of these observations, we make the following definition.

\begin{definition}
  We call a set $S \subseteq \matFq$ \emph{purely core} if
  $S\cap \mcC$ is core for all minimal polynomial classes $\mcC$.
\end{definition}

Any purely core set is core, because unions of core sets are always
core. It is, however, possible that $S$ is core even though some
$S \cap \mcC$ is noncore. For this to occur, $\mcC$ must be an SQD
class. Examples of such behavior can be found in Example~\ref{ex:easy
  SQD} and Section~\ref{section:examples}.

The counts presented in Table~\ref{Noncore table} can be used to
enumerate all the purely core subsets of $\matFq$.  While there exist
core sets that are not purely core, we show in
Section~\ref{section:asymptotics} that the number of such sets is
small in comparison to the number of purely core sets (see
Theorem~\ref{Asymptotic thm}). Furthermore, asymptotically,
\textit{almost all} subsets of $\matFq$ are purely core (and hence are
core):

\begin{reptheorem}{Asymptotic main thm}
\mbox{}
\begin{enumerate}
\item\label{Asymptotic main thm 1} As $q \to \infty$, almost all subsets of $\matFq$ are purely core. That is, as $q \to \infty$,
\begin{equation*}
\dfrac{\phantom{x}\#\{\text{purely core subsets of $\matFq$}\}\phantom{x}}{\#\{\text{subsets of $\matFq$}\}} \to 1.
\end{equation*}

\item\label{Asymptotic main thm 2} As $q \to \infty$, almost all core subsets of $\matFq$ are purely core. That is, as $q \to \infty$,
\begin{equation*}
\dfrac{\phantom{x}\#\{\text{purely core subsets of $\matFq$}\}\phantom{x}}{\#\{\text{core subsets of $\matFq$}\}} \to 1.
\end{equation*}

\item\label{Asymptotic main thm 3} As $q \to \infty$, almost all subsets of $\matFq$ are core. That is, as $q \to \infty$,
\begin{equation*}
\dfrac{\phantom{x}\#\{\text{core subsets of $\matFq$}\}\phantom{x}}{\#\{\text{subsets of $\matFq$}\}} \to 1.
\end{equation*}
\end{enumerate}

\end{reptheorem}

In the last section, we present some examples to demonstrate that
counting core subsets that are not purely core quickly becomes a
challenging task, even for small~$q$.


\section{Preliminaries, Sizes of Similarity Classes, and the Irreducible Cases}
\label{section:prelim_and_sizes}
In this section, we first collect a number of notations and results
that will be used frequently in the rest of the paper. We also provide
proofs for the number of each type of minimal polynomial class listed
in Table \ref{Noncore table}, as well as the size of each
class. Finally, we prove Theorem \ref{theorem:IRR}, which enumerates
the number of core subsets of $\mcC(m)$ when $m$ is irreducible.

Given $S \subseteq M_2(\F_q)$, let $\phi_S$ be the monic least common
multiple in $\F_q[x]$ of all the minimal polynomials of the matrices
in $S$. Since $M_2(\F_q)$ is finite, such a polynomial is guaranteed
to exist. Note that if $S=\varnothing$, then $\phi_S=1$. The
polynomial $\phi_S$ is always an element of
$N(S)$~(\cite[Theorem~4.4(1)]{Werner:2022:null-ideals}). Furthermore,
it is shown in \cite[Corollary~4.6]{Werner:2022:null-ideals} that $S$
is core if and only if $N(S)$ contains no polynomial of degree less
than $\deg \phi_S$. Equivalently, $S$ is core if and only if each
polynomial in $N(S)$ is divisible by $\phi_S$, and this holds if and
only if every minimal polynomial of a matrix in $S$ divides each
polynomial in $N(S)$.

We also obtain the
following characterization of noncore subsets in quadratic minimal
polynomial classes.

\begin{lemma}\label{lem:deg 1}
  Let $\mcC$ be either an IRR, SQR, or SQD class, and let
  $S \subseteq \mcC$ be nonempty. Then, $S$ is noncore if and only if
  $N(S)$ contains a polynomial of degree~$1$.
\end{lemma}
\begin{proof}
  Under the stated hypotheses, $\phi_S$ has degree 2. Hence, $N(S)$ is
  not generated (as a two-sided ideal of $M_2(\F_q)[x])$ by $\phi_S$
  if and only if $N(S)$ contains a linear polynomial.
\end{proof}

This lemma allows us to describe some examples of core and noncore
subsets.

\begin{example}\label{ex:easy SQD}
  Let $A_1 = \smat{1}{0}{0}{0}$, $A_2 = \smat{1}{1}{0}{0}$, and
  $A_3=\smat{0}{0}{0}{1}$. Let $S_1 = \{A_1, A_2\}$ and
  $S_2 = \{A_1, A_3\}$. Then, all three matrices are in the SQD class
  $\mcC(x(x-1))$, and $\phi_{S_1}(x)=\phi_{S_2}(x)=x(x-1)$. The null
  ideal $N(S_1)$ contains the polynomial $\smat{0}{1}{0}{0}x$, and so
  $S_1$ is noncore by Lemma \ref{lem:deg 1}. However, we can prove
  that $S_2$ is core. Suppose that $N(S_2)$ contains
  $\alpha x + \beta$, where $\alpha, \beta \in M_2(\F_q)$. Then,
  $\alpha A_1 + \beta = 0 = \alpha A_3 + \beta$, which implies that
  $\alpha(A_1 - A_3) = 0$. Since $A_1 - A_3$ is invertible, we see
  that $\alpha=0$, and hence $\beta=0$ as well. Consequently, $N(S_2)$
  contains no linear polynomials, and so $S_2$ is core.

  Now, let
  $S_0 = \left\{\smat{0}{0}{0}{0}, \smat{1}{0}{0}{1}\right\}$, let
  $S'$ be any subset of $\mcC(x(x-1))$, and let $T = S_0 \cup
  S'$. Then, $\phi_T(x) = x(x-1)$. Since $T$ contains the scalar
  matrices in $S_0$, $N(T)$ contains no linear polynomials. Indeed if,
  $\alpha x + \beta \in N(T)$, then evaluation at the zero matrix
  shows that $\beta=0$, and evaluation at the identity matrix then
  yields $\alpha=0$. Thus, $T$ is core, even though $S'$ may be chosen
  to be noncore.
\end{example}

Building off of Lemma \ref{lem:deg 1}, we provide a number of
equivalent conditions for subsets of a quadratic minimal polynomial
class to be core. For a subset $S$ of an SQD class, the property of
being core can be described in terms of certain left ideals of
$M_2(\F_q)$ that we call the $\mcL$-modules of $S$.

\begin{definition}[{see~\cite[Definition~6.2]{Werner:2022:null-ideals}}]\label{def:L-modules}
  Let $\mcC((x-a)(x-b))$ be an SQD class, and let
  $S \subseteq \mcC((x-a)(x-b))$. The \textit{$\mcL$-modules} of $S$
  are
  \begin{align*}
    \mcL(S,a) &= \{\alpha \in M_2(\F_q) \mid \alpha(x-a) \in N(S)\} \text{ and }\\
    \mcL(S,b) &= \{\alpha \in M_2(\F_q) \mid \alpha(x-b) \in N(S)\}.
  \end{align*}
\end{definition}

\begin{proposition}\label{Single class prop}
  Let $m \in \F_q[x]$ be a monic quadratic polynomial and let
  $S \subseteq \mcC(m)$ be nonempty.
\begin{enumerate}
\item\label{Single class prop 1}
  \cite[Corollary~5.9(2)]{Werner:2022:null-ideals} If $m$ is IRR, then
  $S$ is core if and only if $|S| \geq 2$.
\item\label{Single class prop 2}
  \cite[Theorem~5.14]{Werner:2022:null-ideals} If $m$ is SQR or SQD,
  then $S$ is core if and only if there exist $A$, $B \in S$ such that
  $A-B$ is invertible.
\item\label{Single class prop 3}
  \cite[Proposition~6.4(3)]{Werner:2022:null-ideals} Assume
  $m(x)=(x-a)(x-b)$ is SQD. Then, $S$ is core if and only if
  $\mcL(S,a)=\mcL(S,b)=\{0\}$.
\end{enumerate}
\end{proposition}

Let $m_1$, \ldots, $m_t$ be all of the possible minimal polynomials of
matrices in $M_2(\F_q)$.  As noted in the introduction, any subset
$S \subseteq M_2(\F_q)$ admits a decomposition
$S = \bigcup_{i=1}^t (S \cap \mcC(m_i))$, and $S$ is said to be purely
core if $S \cap \mcC(m_i)$ is core for each $i$. Example \ref{ex:easy
  SQD} demonstrates that not all core sets are purely core. If $S$ is
such a set, then $S$ must have a nonempty intersection with some SQD
class. Put differently, if we ignore SQD classes, then sets are core
if and only if they are purely core.

\begin{proposition}\label{Purely core non-SQD prop}
  Let $\mcC = \mcC_1 \cup \cdots \cup \mcC_k$ be a union of minimal
  polynomial classes in $\matFq$ such that $\mcC_i$ is not SQD for all
  $1 \leq i \leq k$. Then, a subset of $\mcC$ is core if and only if
  it is purely core.
\end{proposition}
\begin{proof}
  The reverse implication is trivial. For the forward implication, let
  $S \subseteq \mcC$ be core. Then, $S = S_1 \cup \cdots \cup S_k$,
  where $S_i = S \cap \mcC_i$ for each $i$. Fix $j$ between $1$ and
  $k$. If $\mcC_j$ is LIN, then all subsets of $\mcC_j$ are core. If
  $\mcC_j$ is IRR or SQR, then $S_j$ must be core by~\cite[Corollary
  5.4]{Werner:2022:null-ideals} or~\cite[Proposition
  5.12]{Werner:2022:null-ideals}, respectively. Thus, $S_i$ is core
  for all $1 \leq i \leq k$, and therefore $S$ is purely core.
\end{proof}

Next, we count the number of each type (LIN, IRR, SQR, or SQD) of
minimal polynomial class in $M_2(\F_q)$, and determine the size of
each class. These results are not new, but they were not readily
available in the literature, so we provide short proofs for the sake
of completeness.

\begin{lemma}\label{Counting polynomial types lem}
  The polynomial ring $\F_q[x]$ contains $q$ monic linear polynomials,
  $(q^2-q)/2$ monic irreducible quadratic polynomials, $q$ SQR
  polynomials, and $(q^2-q)/2$ SQD polynomials.
\end{lemma}
\begin{proof}
  The polynomials in LIN, SQR, or SQD classes are determined by their
  roots. Hence there are $q$ linear polynomials, and $q$ SQR
  polynomials. In the SQD case, there are $\binom{q}{2} = (q^2-q)/2$
  such polynomials, which correspond to the subsets of $\F_q$ of size
  2. Since $\F_q[x]$ contains $q^2$ monic polynomials of degree 2,
  this leaves $q^2 - q - (q^2-q)/2 = (q^2-q)/2$ monic quadratic
  irreducible polynomials.
\end{proof}

\begin{proposition}\label{Class sizes prop}
  Let $\mcC$ be a minimal polynomial class in $\matFq$.
  \begin{enumerate}
  \item\label{Class sizes prop 1} If $\mcC$ is LIN, then $|\mcC| = 1$.
  \item\label{Class sizes prop 2} If $\mcC$ is IRR, then
    $|\mcC| = q^2-q$.
  \item\label{Class sizes prop 3} If $\mcC$ is SQR, then
    $|\mcC| = q^2-1$.
  \item\label{Class sizes prop 4} If $\mcC$ is SQD, then
    $|\mcC| = q^2+q$.
  \end{enumerate}
\end{proposition}
\begin{proof}
  Part~\eqref{Class sizes prop 1} is trivial. For the degree 2 cases,
  assume that $\mcC = \mcC(m)$, where $m$ is a monic quadratic
  polynomial in $\F_q[x]$, and let $A \in \matFq$ have minimal
  polynomial $m$. Then, $\mcC$ is equal to the similarity class of
  $A$. Since this class is the orbit of $A$ under the conjugation
  action of $\GL2$, this means that $|\mcC| = |\GL2|/|\cent(A)|$,
  where $\cent(A) = \{U \in \GL2 \mid U^{-1}AU=A\}$ is the centralizer
  of $A$. As $|\GL2| = (q^2-1)(q^2-q)$, it suffices to calculate
  $|\cent(A)|$ in each case.

  The minimal polynomial of $A$ equals the characteristic polynomial
  of $A$, so by~\cite[Corollary~4.4.18]{HornJohnson} $\cent(A)$ is
  equal to the unit group of the ring $\F_q[A]$, and
  $\F_q[A] \cong \F_q[x]/(m)$. We now consider cases depending on the
  factorization type of $m$. If $\mcC$ is IRR, then $\F_q[A]$ is a
  field of order $q^2$, $|\cent(A)| = q^2-1$, and $|\mcC| = q^2-q$.

  Suppose next that $\mcC$ is SQR and $m(x)=(x-\lambda)^2$ for some
  $\lambda \in \F_q$. Since the translation
  $f(x) \mapsto f(x-\lambda)$ is a ring automorphism of $\F_q[x]$, we
  have
  \begin{equation*}
    \F_q[A] \cong \F_q[x]/((x-\lambda)^2) \cong \F_q[x]/(x^2).
  \end{equation*}
  This last ring is sometimes called the ring of dual numbers over
  $\F_q$ \cite[p.~729]{DummitFoote}, and is isomorphic to
  $\big\{\smat{a}{b}{0}{a} \mid a, b \in \F_q\big\}$ via the mapping
  $a+bx \mapsto \smat{a}{b}{0}{a}$. From this form of $\F_q[A]$, we
  may calculate that $|\cent(A)| = q(q-1)$ and $|\mcC| = q^2-1$.

  Finally, assume that $\mcC$ is SQD and that $m(x)=(x-a)(x-b)$ for
  some $a,b \in \F_q$ with $a\ne b$. Then, by the Chinese Remainder
  Theorem \cite[Section~7.6, Theorem~17]{DummitFoote},
  \begin{equation*}
    \F_q[A] \cong \F_q[x]/((x-a)(x-b)) \cong \F_q[x]/(x-a) \times \F_q[x]/(x-b) \cong \F_q \times \F_q.
  \end{equation*}
  It follows that $|\cent(A)| = (q-1)^2$ and $|\mcC| = q(q+1)$.
\end{proof}

When $m$ is an irreducible polynomial, it is not hard to determine the
number of core subsets of $\mcC(m)$.

\begin{theorem}\label{theorem:IRR}
  Let $m$ be an irreducible polynomial of degree at most $2$.
  \begin{enumerate}
  \item If $m$ is linear, then $\mcC(m)$ contains $2$ core subsets.
  \item If $m$ is quadratic, then $\mcC(m)$ contains $q^2-q$ noncore
    subsets.
  \end{enumerate}
\end{theorem}
\begin{proof}
  If $m(x) = x-a$ with $a\in \F_q$, then $\mcC(m) = \{a\}$. This is a
  subset of the center of $\matFq$ and hence core,
  cf.~\cite[Lemma~3.3]{Werner:2022:null-ideals}.

  Let $m$ be quadratic irreducible and $S\subseteq \mcC(m)$. By
  Proposition~\ref{Single class prop}\eqref{Single class prop 1}, $S$
  is core if and only if $|S| \ge 2$. Since $|\mcC(m)| = q^2-q$ by
  Proposition~\ref{Class sizes prop}\eqref{Class sizes prop 2},
  $\mcC(m)$ contains $q^2-q$ singleton sets which are exactly the
  noncore subsets.
\end{proof}

In the next two sections, we will focus on counting noncore subsets of
SQR classes and SQD classes. Dealing with SQR classes is
straightforward, but SQD classes require a more detailed analysis and
make frequent use of the $\mcL$-modules from
Definition~\ref{def:L-modules}.


\section{SQR polynomials}
\label{section:SQR}
Consider an SQR class $\mcC((x-a)^2)$, where $a \in \F_q$. Matrices in
$\mcC((x-a)^2)$ are in bijective correspondence with matrices in
$\mcC(x^2)$ via the translation $A \to A+a$, and this mapping
preserves core subsets~\cite[Proposition
3.4]{Werner:2022:null-ideals}. Thus, to determine the number of core
subsets of $\mcC((x-a)^2)$, it suffices to work in $\mcC(x^2)$.

\begin{definition}\label{SQR subclasses}
  In $\mcC(x^2)$, we define the following subclasses of matrices.
  \begin{itemize}
  \item $T_0 = \left\{\smat{0}{a}{0}{0} \mid a \ne 0\right\}$
  \item $T_0' = \left\{\smat{0}{0}{a}{0} \mid a \ne 0\right\}$
  \item For each $\lambda \in \F_q^\times$,
    $T_\lambda = \left\{ \smatb{a}{\phantom{xx}\lambda
        a}{-\lambda^{-1}a}{\phantom{xx}-a} \longmid a \ne 0 \right\}$
  \end{itemize}
  We let $\msP = \{T_0, T_0'\} \cup \{T_\lambda \mid \lambda \ne 0\}$
  be the collection of all of the above subclasses.
\end{definition}

\begin{proposition}\label{SQR prop} \mbox{}
  \begin{enumerate}
  \item\label{SQR prop 1} For all $T \in \msP$, $|T| = q-1$.
  \item\label{SQR prop 2} The sets in $\msP$ form a partition of
    $\mcC(x^2)$.
  \item\label{SQR prop 3} For all $A$, $B \in \mcC(x^2)$, $A-B$ is
    singular if and only if there exists $T \in \msP$ such that
    $A, B \in T$.
  \item\label{SQR prop 4} Let $S \subseteq \mcC(x^2)$. Then, $S$ is
    noncore if and only if $S$ is a nonempty subset of some
    $T \in \msP$.
  \end{enumerate}
\end{proposition}
\begin{proof}
  \eqref{SQR prop 1} Each set in $\msP$ consists of the nonzero scalar
  multiples of a single matrix. Explicitly,
  $T_0 = \F_q^\times \cdot \smat{0}{1}{0}{0}$,
  $T_0' = \F_q^\times \cdot \smat{0}{0}{1}{0}$, and
  $T_\lambda = \F_q^\times \cdot
  \smat{1}{\lambda}{-\lambda^{-1}}{-1}$. Thus, each set in $\msP$ has
  cardinality $q-1$.

  \eqref{SQR prop 2} Let $\mcU = \bigcup_{T \in \msP} T$ be the union
  of all the sets in $\msP$.  Given $S$, $T \in \msP$ such that
  $S \ne T$, it is clear that $S \cap T = \varnothing$. Thus,
  $|\mcU| = (q+1)(q-1) = |\mcC(x^2)|$.

  \eqref{SQR prop 3} $(\Leftarrow)$ Assume that matrices $A$ and $B$
  are contained in the same set $T \in \msP$. Since elements of $T$
  are scalar multiples of one another, we have $B=cA$ for some nonzero
  $c \in \F_q$. Thus, $A-B = (1-c)B$, which is singular because $B$ is
  singular.

  ($\Rightarrow$) Let $A, B \in \mcC(x^2)$ and assume that $A-B$ is
  singular. By (2), there exist $T_1$, $T_2 \in \msP$ such that
  $A \in T_1$ and $B \in T_2$. Since both $A$ and $B$ have rank one
  and $A-B$ is singular, it follows that $T_1=T_2$.

  \eqref{SQR prop 4} The empty set is core, so assume that
  $S \ne \varnothing$. By Proposition~\ref{Single class
    prop}\eqref{Single class prop 2}, $S$ is noncore if and only if
  $A-B$ is singular for all $A$, $B \in S$. By part~\eqref{SQR prop
    3}, this holds if and only if $S \subseteq T$ for some
  $T \in \msP$.
\end{proof}

\begin{theorem}\label{SQR thm}
  Each SQR class in $\matFq$ contains $(q+1)(2^{q-1}-1)$ noncore
  subsets.
\end{theorem}
\begin{proof}
  Let $\mcC$ be an SQR class in $\matFq$. As noted at the start of
  Section~\ref{section:SQR}, the number of noncore subsets of $\mcC$
  is the same as the number of noncore subsets of $\mcC(x^2)$. By
  Proposition~\ref{SQR prop}\eqref{SQR prop 4}, the noncore subsets of
  $\mcC(x^2)$ are exactly the nonempty subsets of the subclasses
  introduced in Definition~\ref{SQR subclasses}. There are $q+1$ such
  subclasses, each of size $q-1$. Hence, $\mcC(x^2)$ contains
  $(q+1)(2^{q-1}-1)$ noncore subsets.
\end{proof}

Via Proposition \ref{Single class prop}\eqref{Single class prop 2}, we
can restate Theorem \ref{SQR thm} in terms of subsets where the
difference of each pair of matrices is singular.
\begin{corollary}\label{SQR cor}
  Let $m$ be an SQR polynomial. Then, $\mcC(m)$ contains
  $(q+1)(2^{q-1}-1)$ non-empty subsets $S$ with the property that $A-B$
  is singular for all $A$, $B \in S$.
\end{corollary}


\section{SQD polynomials}
\label{section:SQD}
Among the four types of minimal polynomial classes (LIN, IRR, SQR, and
SQD), the SQD classes are the most difficult in which to count core
subsets. When $\mcC$ is an SQR class, Proposition \ref{SQR prop} shows
that we can partition $\mcC$ in such a way that noncore subsets are
easy to identify. We are not able to obtain such a clean decomposition
in an SQD class. However, by using $\mcL$-modules, we can define a
useful collection of noncore subsets of an SQD class. Recall that for
a subset $S \subseteq \mcC((x-a)(x-b))$ of an SQD class, the
$\mcL$-modules of $S$ are
\begin{align*}
  \mcL(S,a) &= \{\alpha \in M_2(\F_q) \mid \alpha(x-a) \in N(S)\} \text{ and }\\
  \mcL(S,b) &= \{\alpha \in M_2(\F_q) \mid \alpha(x-b) \in N(S)\}.
\end{align*}

\begin{definition}\label{Bubble def}
  Let $\mcC = \mcC((x-a)(x-b))$ be an SQD class. For each
  $A \in \mcC$, we define
  \begin{align*}
    \mcB(A,a) &:= \{B \in \mcC \mid \mcL(\{A,B\},a) \ne \{0\}\} \text{ and }\\
    \mcB(A,b) &:= \{B \in \mcC \mid \mcL(\{A,B\},b) \ne \{0\}\}.
  \end{align*}
  We call these the \emph{$\mcB$-sets} for $A$.
\end{definition}

In Lemmas~\ref{L-module lem 2} and~\ref{Bset lem} below,
we prove that each $\mcB$-set is noncore and has size $q$. By
\cite[Corollary 5.9]{Werner:2022:null-ideals}, any subset of an SQD
class of size at least $q+1$ is core. Hence, $\mcB$-sets are noncore
subsets of maximal size in an SQD class. In Theorem~\ref{theorem:SQD},
$\mcB$-sets will be used to describe all of the noncore subsets of the
SQD class $\mcC((x-a)(x-b))$. Since these $\mcB$-sets are defined in
terms of $\mcL$-modules, we require more precise knowledge of these
structures. First, we show that the nonzero $\mcL$-modules in $\matFq$
can be indexed by using linear subspaces and row vectors in~$\F_q^2$.

\begin{definition}\label{L_v def}
  Let $v = [v_1, v_2 ]$ be a row vector in $\F_q^2$. We
  define
  \begin{equation*}
    \mcL_v = \mcL_{[ v_1, v_2 ]} := \left\{\begin{bmatrix} y \\ z \end{bmatrix} \big[\begin{matrix}v_1 & v_2 \end{matrix}\big] \longmid  y, z \in \F_q\right\}.
  \end{equation*}
\end{definition}

\begin{lemma}\label{L-module lem 1} \mbox{}
  \begin{enumerate}
  \item\label{L-mod 1, 1} For all $v \in \F_q^2$, $\mcL_v$ is a left
    ideal of $\matFq$. When $v \ne 0$, $\mcL_v$ is a minimal left
    ideal of $\matFq$.

  \item\label{L-mod 1, 2} For all $v$, $w \in \F_q^2$,
    $\mcL_v = \mcL_w$ if and only if $v$ and $w$ are nonzero scalar
    multiples of one another.

  \item\label{L-mod 1, 3} If $L$ is a minimal left ideal of $\matFq$,
    then either $L = \mcL_{[ 0,1 ]}$ or
    $L = \mcL_{[ 1,\lambda ]}$ for some
    $\lambda \in \F_q$.
  \end{enumerate}
\end{lemma}
\begin{proof}
  \eqref{L-mod 1, 1} Each set $\mcL_v$ is easily seen to be a left
  ideal of $\matFq$. When $v \ne 0$, $\mcL_v$ is nonzero and proper,
  and hence must be a minimal left ideal of $\matFq$.

  \eqref{L-mod 1, 2} $(\Rightarrow)$ The result is trivial if $v=0$ or
  $w=0$.  So, assume that $v=[ v_1, v_2]$ and
  $w=[ w_1, w_2 ]$ are both nonzero, and that
  $\mcL_v=\mcL_w$. Taking $y=1$ and $z=0$ in the definition of
  $\mcL_w$, we see that $\smat{w_1}{w_2}{0}{0} \in L_w = L_v$. Hence,
  $w_1 = \lambda v_1$ and $w_2 = \lambda v_2$ for some nonzero
  $\lambda \in \F_q$.

  $(\Leftarrow$) Assume that $v$, $w \in \F_q^2$ and $w = \lambda v$
  for some nonzero $\lambda \in \F_q$. Given $A \in \mcL_v$, we
  clearly have $\lambda A \in \mcL_w$. But, $\mcL_w$ is a left ideal,
  so $A = \tfrac{1}{\lambda}(\lambda A) \in \mcL_w$. Thus,
  $\mcL_v \subseteq \mcL_w$. The proof that $\mcL_w \subseteq \mcL_v$
  is similar.

  \eqref{L-mod 1, 3} Note that the left ideals of $\matFq$ correspond
  to subspaces of $\F_q^2$. Explicitly, in any left ideal $L$, the
  rows of $L$ comprise a subspace of $\F_q^2$. Conversely, given a
  subspace $V$ of $\F_q^2$, the set of matrices in $M_2(\F_q)$ whose
  rows are vectors in $V$ forms a left ideal. In particular, the
  minimal left ideals of $\matFq$ are in one-to-one correspondence
  with the linear subspaces of $\F_q^2$.  There are $q+1$ such
  subspaces, namely, $\mcL_{[ 0,1 ]}$ and
  $\mcL_{[ 1,\lambda ]}$ for $\lambda\in \F_q$.  The
  result follows.
\end{proof}

Next, we prove some basic properties about $\mcB$-sets and their
relation to $\mcL$-modules.

\begin{lemma}\label{L-module lem 2}
  Let $\mcC = \mcC((x-a)(x-b))$ be an SQD class and let $A \in \mcC$.
  \begin{enumerate}
  \item\label{L-mod 2, 1} $\mcL(\{A\},a) = \mcL_v$ for some
    $v \in \{[ 0,1 ]\} \cup \{[ 1, \lambda ]
    \mid \lambda \in \F_q\}$. In particular, there are $q+1$
    possibilities for this $\mcL$-module.

  \item\label{L-mod 2, 2}
    $\mcB(A,a) = \{B \in \mcC \mid \mcL(\{B\},a) =
    \mcL(\{A\},a)\}$. Thus, for all $A$, $B \in \mcC$,
    $\mcB(A,a) = \mcB(B,a)$ if and only if
    $\mcL(\{A\},a) = \mcL(\{B\},a)$.

  \item\label{L-mod 2, 3} $\mcB(A,a)$ is noncore.

  \item\label{L-mod 2, 4} Let $U \in \GL2$. Then,
    $\mcB(U^{-1}AU,a) = U^{-1}\mcB(A,a)U$.
\end{enumerate}
Moreover, each of the above statements holds when $\mcB(A,a)$ is
replaced with $\mcB(A,b)$.
\end{lemma}
\begin{proof}
  \eqref{L-mod 2, 1} By~\cite[Lemma 6.3(1)]{Werner:2022:null-ideals}
  $\mcL(\{A\},a)$ is nonzero and is a minimal left ideal
  of~$\matFq$. The result now follows from Lemma~\ref{L-module lem
    1}\eqref{L-mod 1, 3}.

  \eqref{L-mod 2, 2} By part~\eqref{L-mod 2, 1}, there exists a
  nonzero $v \in \F_q^2$ such that $\mcL(\{A\},a) = \mcL_v$. Take
  $B \in \mcB(A,a)$. Then $\mcL(\{A,B\},a) \ne \{0\}$,   
  $\mcL(\{A,B\},a)$ is a left ideal of~$\matFq$ 
  (\cite[Lemma~6.3(1)]{Werner:2022:null-ideals}), and 
  $\mcL(\{A,B\},a) \subseteq \mcL(\{A\},a)$ which implies that
  $\mcL(\{A,B\},a) = \mcL_v$. Since
  $\mcL(\{A,B\},a) \subseteq \mcL(\{B\},a)$ also holds and
  $\mcL(\{B\},a)$ is a minimal left ideal, it follows
  $\mcL(\{B\},a) = \mcL_v$.  Conversely, if $C \in \mcC$ with
  $\mcL(\{C\},a) = \mcL_v$, then
  $\mcL(\{C\},a) = \mcL_v = \mcL(\{A\},a)$, and all of these modules
  are nonzero. Thus, $\mcL(\{A,C\},a) \ne \{0\}$ and
  $C \in \mcB(A,a)$.

  \eqref{L-mod 2, 3} Let $S = \mcB(A,a)$ and let
  $\alpha \in \mcL(\{A\},a)$ be nonzero. By \eqref{L-mod 2, 2},
  $\alpha \in \mcL(\{B\},a)$ for all $B \in S$. Thus,
  $\alpha(x-a) \in N(S)$ and $S$ is noncore by Lemma \ref{lem:deg 1}.

  \eqref{L-mod 2, 4} Let $B \in U^{-1}\mcB(A,a)U$. Then,
  $B = U^{-1}CU$ for some $C \in \mcB(A,a)$. By part~\eqref{L-mod 2,
    2}, we have
  \begin{equation*}
    \mcL(\{A\},a) = \mcL(\{C\},a) = \mcL(\{UBU^{-1}\},a).
  \end{equation*}
  By~\cite[Lemma 6.3(2)]{Werner:2022:null-ideals},
  $\mcL(\{UBU^{-1}\},a) = U\mcL(\{B\},a)U^{-1}$. This further implies
  that $\mcL(\{B\},a) = \mcL(\{U^{-1}AU\},a)$. Hence,
  $B \in \mcB(U^{-1}AU,a)$ and
  $U^{-1}\mcB(A,a)U \subseteq \mcB(U^{-1}AU,a)$. The proof of the
  reverse inclusion is similar.
\end{proof}

The next example demonstrates one way to visualize $\mcB$-sets in an
SQD class, and how they can be used to count noncore subsets.

\begin{example}\label{F_2 bubble set ex}
  Let $\mcC = \mcC(x(x+1))$ in $\mat{2}$. Then, $\mcC$ consists of the
  following six matrices:
  \begin{align*}
    A_1 &=\smat{0}{0}{0}{1}, \quad A_2 = \smat{0}{1}{0}{1}, \quad A_3 = \smat{0}{0}{1}{1}, \\
    A_4 &= \smat{1}{0}{0}{0}, \quad A_5 = \smat{1}{1}{0}{0}, \quad A_6 = \smat{1}{0}{1}{0}.
  \end{align*}

  Computing $\mcB(A,0)$ and $\mcB(A,1)$ as $A$ runs along $\mcC$, we
  find that $|\mcB(A,0)| = |\mcB(A,1)| = 2$ and
  $\mcB(A,0) \cap \mcB(A,1) = \{A\}$. We can represent this situation
  pictorially in Figure~\ref{F_2 bubble figure}. In that picture, each
  oval corresponds to a $\mcB$-set, and the associated $\mcL$-module
  is shown alongside the $\mcB$-set.

\newcommand{\bubble}[3]
{(#1) ellipse (#2 and #3)}

\begin{figure}[h]
  \centering
  \begin{tikzpicture}[scale=1]
\node at (0,5) {$\smat{0}{0}{0}{1}$};
\node at (-2,5) {$\smat{0}{0}{1}{1}$};
\node at (2,5) {$\smat{0}{1}{0}{1}$};
\draw \bubble{1,5}{2}{0.5};
\draw \bubble{-1,5}{2}{0.5};
\node at (-4.5,5) {\begin{tabular}{r} $\mcB(A_1,0)$\\$\mcL(-,0) = \mcL_{[ 1,0]}$\end{tabular}};
\node at (4.5,5) {\begin{tabular}{l} $\mcB(A_1,1)$\\$\mcL(-,1) = \mcL_{[ 0, 1]}$\end{tabular}};

\node at (0,3) {$\smat{0}{1}{0}{1}$};
\node at (-2,3) {$\smat{1}{0}{1}{0}$};
\node at (2,3) {$\smat{0}{0}{0}{1}$};
\draw \bubble{1,3}{2}{0.5};
\draw \bubble{-1,3}{2}{0.5};
\node at (-4.5,3) {\begin{tabular}{r} $\mcB(A_2,0)$\\$\mcL(-,0) = \mcL_{[ 1,1]}$\end{tabular}};
\node at (4.5,3) {\begin{tabular}{l} $\mcB(A_2,1)$\\$\mcL(-,1) = \mcL_{[ 0, 1]}$\end{tabular}};

\node at (0,1) {$\smat{0}{0}{1}{1}$};
\node at (-2,1) {$\smat{0}{0}{0}{1}$};
\node at (2,1) {$\smat{1}{1}{0}{0}$};
\draw \bubble{1,1}{2}{0.5};
\draw \bubble{-1,1}{2}{0.5};
\node at (-4.5,1) {\begin{tabular}{r} $\mcB(A_3,0)$\\$\mcL(-,0) = \mcL_{[ 1,0]}$\end{tabular}};
\node at (4.5,1) {\begin{tabular}{l} $\mcB(A_3,1)$\\$\mcL(-,1) = \mcL_{[ 1, 1]}$\end{tabular}};

\node at (0,-1) {$\smat{1}{0}{0}{0}$};
\node at (-2,-1) {$\smat{1}{1}{0}{0}$};
\node at (2,-1) {$\smat{1}{0}{1}{0}$};
\draw \bubble{1,-1}{2}{0.5};
\draw \bubble{-1,-1}{2}{0.5};
\node at (-4.5,-1) {\begin{tabular}{r} $\mcB(A_4,0)$\\$\mcL(-,0) = \mcL_{[ 0,1]}$\end{tabular}};
\node at (4.5,-1) {\begin{tabular}{l} $\mcB(A_4,1)$\\$\mcL(-,1) = \mcL_{[ 1, 0]}$\end{tabular}};

\node at (0,-3) {$\smat{1}{1}{0}{0}$};
\node at (-2,-3) {$\smat{1}{0}{0}{0}$};
\node at (2,-3) {$\smat{0}{0}{1}{1}$};
\draw \bubble{1,-3}{2}{0.5};
\draw \bubble{-1,-3}{2}{0.5};
\node at (-4.5,-3) {\begin{tabular}{r} $\mcB(A_5,0)$\\$\mcL(-,0) = \mcL_{[ 0,1]}$\end{tabular}};
\node at (4.5,-3) {\begin{tabular}{l} $\mcB(A_5,1)$\\$\mcL(-,1) = \mcL_{[ 1, 1]}$\end{tabular}};

\node at (0,-5) {$\smat{1}{0}{1}{0}$};
\node at (-2,-5) {$\smat{0}{1}{0}{1}$};
\node at (2,-5) {$\smat{1}{0}{0}{0}$};
\draw \bubble{1,-5}{2}{0.5};
\draw \bubble{-1,-5}{2}{0.5};
\node at (-4.5,-5) {\begin{tabular}{r} $\mcB(A_6,0)$\\$\mcL(-,0) = \mcL_{[ 1,1]}$\end{tabular}};
\node at (4.5,-5) {\begin{tabular}{l} $\mcB(A_6,1)$\\$\mcL(-,1) = \mcL_{[ 1, 0]}$\end{tabular}};
\end{tikzpicture}
  \caption{$\mcB$-sets for matrices in
    $\mcC(x(x+1)) \subseteq \mat{2}$}
  \label{F_2 bubble figure}
\end{figure}

There is a great deal of symmetry present in this picture. Note that
there are three nonzero $\mcL$-modules in $M_2(\F_2)$. Each
$\mcL$-module occurs four times in the Figure~\ref{F_2 bubble figure},
twice corresponding to $\mcB$-sets for the root 0, and twice
corresponding to $\mcB$-sets for the root 1. We can also use this
picture to identify and count noncore subsets of $\mcC$. For instance,
if $A$, $B \in \mcC$ with $A-B$ singular, then
$B \in \mcB(A,0) \cup \mcB(A,1)$. However, if
$B \in \mcB(A,0) \setminus\{A\}$ and $C \in \mcB(A,1)\setminus\{A\}$,
then $B-C$ is invertible. This suggests that each noncore subset of
$\mcC$ is contained in a $\mcB$-set. Moreover, each $\mcB$-set is
noncore by Lemma \ref{L-module lem 2}\eqref{L-mod 2, 3}. There are six
unique $\mcB$-sets---three of the form $\mcB(-,0)$ and three of the
form $\mcB(-,1)$---which account for all of the noncore subsets of
$\mcC$ of size 2. Along with the six singleton subsets, we see that
there are 12 noncore subsets of $\mcC$.
\end{example}

The behavior seen in Example~\ref{F_2 bubble set ex} is typical of SQD
classes in $\matFq$. Ultimately, we will prove that noncore subsets in
an SQD class are contained in $\mcB$-sets, and the symmetry among
$\mcB$-sets can be illustrated as in Figure~\ref{F_2 bubble
  figure}. The next several results build up the theory required to
prove these claims.

\begin{lemma}\label{Bset lem}
  Let $\mcC = \mcC((x-a)(x-b))$ be an SQD class in $\matFq$.
  \begin{enumerate}
  \item\label{Bset lem 1} For all $A \in \mcC$,
    $\mcB(A,a) \cap \mcB(A,b) = \{A\}$.

  \item\label{Bset lem 2} For all $A \in \mcC$,
    $|\mcB(A,a)| = |\mcB(A,b)| = q$.

  \item\label{Bset lem 3} Let $A$, $B \in \mcC$. If $B-A$ is singular,
    then $B \in \mcB(A,a) \cup \mcB(A,b)$.

  \item\label{Bset lem 4} Let $A \in \mcC$. If
    $B \in \mcB(A,a) \setminus \{A\}$ and
    $C \in \mcB(A,b) \setminus \{A\}$, then $B-C$ is invertible.

  \item\label{Bset lem 5} $\mcC$ contains $q+1$ distinct $\mcB$-sets
    of the form $\mcB(-,a)$, which comprise a partition of $\mcC$. The
    analogous result holds for $\mcB$-sets of the form $\mcB(-,b)$.
  \end{enumerate}
\end{lemma}
\begin{proof}
  \eqref{Bset lem 1} Let $A \in \mcC$ and suppose that there exists
  $B \in \mcB(A,a) \cap \mcB(A,b)$ such that $B \ne A$. Then,
  $\mcL(\{A,B\},a) \ne \{0\}$ and $\mcL(\{A,B\},b) \ne \{0\}$. This
  contradicts~\cite[Lemma 6.3(4)]{Werner:2022:null-ideals}. Hence, no
  such $B$ exists.

  \eqref{Bset lem 2} By Lemma~\ref{L-module lem 2}\eqref{L-mod 2, 4},
  all of the $\mcB$-sets $\mcB(A,a)$ are conjugate as $A$ runs
  through~$\mcC$. Thus, we may assume without loss of generality that
  $A = \smat{a}{0}{0}{b}$. It is straightforward to check that
  $\mcL(\{A\},a) = \mcL_{[ 1, 0 ]}$ and
  $\mcL(\{A\},b) = \mcL_{[ 0, 1 ]}$, and so
  \begin{equation}\label{diagonal bubbles eq}
    \mcB(A,a) = \left\{\smat{a}{0}{c}{b} \longmid c\in \F_q\right\} \text{ and } \mcB(A,b) = \left\{\smat{a}{c}{0}{b} \longmid c\in \F_q\right\}.
  \end{equation}
  Thus, $|\mcB(A,a)| = q = |\mcB(A,b)|$.

  \eqref{Bset lem 3} As in part~\eqref{Bset lem 2}, we may assume that
  $A = \smat{a}{0}{0}{b}$. Let $B = \smat{t}{y}{z}{w} \in \mcC$ be
  such that $B-A$ is not invertible. Since $B \in \mcC$, we have
  $t+w = \tr(B) = \tr(A) = a+b$ and $tw-yz = \det(B) = \det(A) =
  ab$. Hence, $w-b = a-t$ and $yz = tw-ab = t(a+b-t)-ab$ hold. A
  routine calculation now shows that $\det(B-A) = (t-a)(a-b)$. Since
  $\mcC$ is SQD, $a \ne b$. So, in order for $B-A$ to be singular, we
  must have $t=a$. This forces $w=b$ and either $y=0$ or
  $z=0$. From~Equation~\eqref{diagonal bubbles eq}, we see that
  $B \in \mcB(A,a) \cup \mcB(A,b)$.

  \eqref{Bset lem 4} Once again, we may assume that
  $A = \smat{a}{0}{0}{b}$. Given $B \in \mcB(A,a) \setminus \{A\}$ and
  $C \in \mcB(A,a) \setminus \{A\}$, by~\eqref{diagonal bubbles eq},
  we see that $B-C = \smat{0}{c}{d}{0}$ for some nonzero $c$,
  $d \in \F_q$. Thus, $B-C$ is invertible.

  \eqref{Bset lem 5} Every matrix $A\in \mcC$ is contained in the
  $\mcB$-set $\mcB(A,a)$ which is a set of size~$q$ by
  part~\eqref{Bset lem 2}. For two matrices $A$, $B\in \mcC$, it holds
  that $\mcB(A,a) = \mcB(B,a)$ if and only if
  $\mcL(\{A\},a) = \mcL(\{B\},a)$ by Lemma~\ref{L-module lem
    2}\eqref{L-mod 2, 2}. From Lemma~\ref{L-module lem 2}\eqref{L-mod
    2, 1}, we know that there are $q+1$ possible choices for
  $\mcL(\{A\},a)$. Thus, $\mcC$ contains $q+1$ distinct $\mcB$-sets
  $\mcB(-,a)$; note that these sets form a partition of $\mcC$,
  because $|\mcC| = q^2+q$ by Proposition~\ref{Class sizes
    prop}\eqref{Class sizes prop 4}. Likewise, $\mcC$ contains $q+1$
  distinct $\mcB$-sets $\mcB(-,b)$, which also partition the class.
\end{proof}

\begin{proposition}\label{Bset prop}
  Let $\mcC = \mcC((x-a)(x-b))$ be an SQD class in $\matFq$.
  \begin{enumerate}
  \item\label{Bset prop 1} For all $A \in \mcC$, if
    $S \subseteq \mcB(A,a)$ and $S \ne \varnothing$, then $S$ is
    noncore.

  \item\label{Bset prop 2} Let $S \subseteq \mcC$ be noncore. If
    $A \in S$, then $S \subseteq \mcB(A,a)$ or
    $S \subseteq \mcB(A,b)$.

  \item\label{Bset prop 3} Let $S \subseteq \mcC$ be nonempty. Then,
    $S$ is noncore if and only if, for each $A \in S$,
    $S \subseteq \mcB(A,a)$ or $S \subseteq \mcB(A,b)$.

  \item\label{Bset prop 4} Assume that $S \subseteq \mcC$ is noncore
    and $|S| \geq 2$. Then, for each $A \in S$, either
    $S \subseteq \mcB(A,a)$ or $S \subseteq \mcB(A,b)$, and these
    cases are mutually exclusive.
  \end{enumerate}
\end{proposition}
\begin{proof}
  \eqref{Bset prop 1} Fix $A \in \mcC$ and let $S$ be a nonempty
  subset of $\mcB(A,a)$. Let $\alpha \in \mcL(\{A\},a)$ be
  nonzero. Then, for each $B \in S$, $\alpha(B-a) = 0$. Thus, the
  polynomial $\alpha(x-a)$ is in $N(S)$, and $N(S)$ is noncore
  by Lemma \ref{lem:deg 1}.

  \eqref{Bset prop 2} Let $A \in S$. Since $S$ is noncore, $B-C$ is
  singular for all $B$, $C \in S$ by Proposition~\ref{Single class
    prop}\eqref{Single class prop 2}. In particular, $B-A$ is singular
  for each $B \in S$. By Lemma~\ref{Bset lem}\eqref{Bset lem 3},
  $S \subseteq \mcB(A,a) \cup \mcB(A,b)$. However, by Lemma~\ref{Bset
    lem}\eqref{Bset lem 4}, $S$ must be entirely contained in either
  $\mcB(A,a)$ or $\mcB(A,b)$.

  \eqref{Bset prop 3} This follows from~\eqref{Bset prop 1}
  and~\eqref{Bset prop 2}.

  \eqref{Bset prop 4} Let $A \in S$ and consider
  $T = S \setminus\{A\}$. If $T$ contains both a matrix
  $B \in \mcB(A,a)$ and a matrix $C \in \mcB(A,b)$, then $B-C$ is
  invertible by Lemma~\ref{Bset lem}\eqref{Bset lem 4}. Then, $S$ is
  core by Proposition~\ref{Single class prop}\eqref{Single class prop
    2}, a contradiction. Note that the assumption $|S| \geq 2$ is
  necessary here, since $\mcB(A,a) \cap \mcB(A,b) = \{A\}$.
\end{proof}

\begin{theorem}\label{theorem:SQD}
  Each SQD class in $\matFq$ contains $(q+1)(2^{q+1}-q-2)$ noncore
  subsets.
\end{theorem}
\begin{proof}
  Let $S \subseteq \mcC$ be noncore. Assume first that $|S| \geq
  2$. By Proposition~\ref{Bset prop}\eqref{Bset prop 4}, $S$ is a
  subset of either some $\mcB(-,a)$ or some $\mcB(-,b)$, but not
  both. From Lemma~\ref{Bset lem}\eqref{Bset lem 2}, we know that each
  $\mcB$-set has size $q$, and hence contains $2^q-(q+1)$ subsets of
  cardinality at least 2. By Lemma~\ref{Bset lem}\eqref{Bset lem 5},
  there are $q+1$ $\mcB$-sets of the form $\mcB(-,a)$, and $q+1$ of
  the form $\mcB(-,b)$. Thus, there are $2(q+1)(2^q-(q+1))$ possible
  choices for $S$ when $|S| \geq 2$.

  Finally, each singleton subset of $\mcC$ is noncore,
  and there are
  $q^2+q$ such subsets. Thus, in total, the number of noncore subsets
  of $\mcC$ is equal to
  \begin{equation*}
    q^2+q+2(q+1)(2^q-(q+1)) = (q+1)(2^{q+1}-q-2). \qedhere
  \end{equation*}
\end{proof}

As in an SQR class, we can use Proposition \ref{Single class
  prop}\eqref{Single class prop 2} to give an interpretation of
Theorem \ref{theorem:SQD} that is independent of core sets.

\begin{corollary}\label{SQD cor}
  Let $m$ be an SQD polynomial. Then, $\mcC(m)$ contains
  $(q+1)(2^{q+1}-q-2)$ subsets $S$ with the property that $A-B$ is
  singular for all $A$, $B \in S$.
\end{corollary}


\section{Asymptotic Results}
\label{section:asymptotics}
Our results so far allow us to determine the number of purely core
subsets in $M_2(\F_q)$. This still leaves the matter of core sets that
are not purely core. In this section, we show that the number of such
core sets is always small in comparison to the number of purely core
sets. Moreover, as $q \to \infty$, we prove that almost all subsets of
$M_2(\F_q)$ are purely core. We begin by establishing some
inequalities related to the probability that a subset of an SQD class
is core.

\begin{lemma}\label{1/2^x lem}
  For all $x \geq 1$,
  $\Big(1 - \dfrac{1}{4^x}\Big)^{-x} < 1 + \dfrac{1}{2^x}$.
\end{lemma}
\begin{proof}
  It suffices to prove that
  $-x\ln(1 - \tfrac{1}{4^x}) < \ln(1 + \tfrac{1}{2^x})$ for all
  $x \geq 1$. Applying the Taylor series expansions for $\ln(1-t)$ and
  $\ln(1+t)$, we obtain the equivalent inequality
\begin{equation}\label{Taylor series inequality}
  x\sum_{n=1}^\infty \dfrac{1}{n}\Big(\dfrac{1}{4^x}\Big)^n < \sum_{n=1}^\infty (-1)^{n+1}\dfrac{1}{n} \Big(\dfrac{1}{2^x}\Big)^n.
\end{equation}
Note that for a fixed $x$, both series converge because
$|\tfrac{1}{4^x}| < 1$ and $|\tfrac{1}{2^x}| < 1$. So, to
establish~\eqref{Taylor series inequality}, it is enough to show that
for each odd integer $n\geq 1$,
\begin{equation*}
  x\left(\dfrac{1}{n}\left(\dfrac{1}{4^x}\right)^n + \dfrac{1}{n+1}\left(\dfrac{1}{4^x}\right)^{n+1}\right) < \dfrac{1}{n}\left(\dfrac{1}{2^x}\right)^n - \dfrac{1}{n+1}\left(\dfrac{1}{2^x}\right)^{n+1}.
\end{equation*}
Adding $\tfrac{1}{n+1}(\tfrac{1}{2^x})^{n+1}$ to both sides of this
inequality and then multiplying by $(4^x)^n$ produces
\begin{equation}\label{n odd inequality}
  x\left(\dfrac{1}{n} + \dfrac{1}{n+1}\left(\dfrac{1}{4^x}\right)\right) + \dfrac{1}{n+1}\cdot(2^x)^{n-1} < \dfrac{1}{n}\cdot(2^x)^n.
\end{equation}
If $n=1$, then~\eqref{n odd inequality} is satisfied since
\begin{equation*}
  x\left(1 + \dfrac{1}{2}\left(\dfrac{1}{4^x}\right)\right) + \dfrac{1}{2} \leq \dfrac{9}{8}x+\dfrac{1}{2} < 2^x.
\end{equation*}
If $n\ge 2$, then~\eqref{n odd inequality} is satisfied since the
left-hand side is less than
$\tfrac{1}{n}\left(x(1+\tfrac{1}{4^x}\right) + (2^x)^{n-1})$, and
\begin{equation*}
  x\left(1+\dfrac{1}{4^x}\right) \leq \dfrac{5}{4}x < 2^x.
\end{equation*}
This gives
\begin{equation*}
  x\left(1+\dfrac{1}{4^x}\right) + (2^x)^{n-1} \leq \dfrac{5}{4}x + (2^x)^{n-1} <  (2^x)^{n-1} + (2^x)^{n-1} = (2^x)^n.
\end{equation*}

\end{proof}

\begin{lemma}\label{Asymptotic lem}
  Let $\rho = 1 - \dfrac{(q+1)(2^{q+1}-q-2)}{2^{q^2+q}}$.
  \begin{enumerate}
  \item\label{Asymptotic lem 1} $\rho$ is the probability that a
    subset selected uniformly at random from an SQD class is core.
  \item\label{Asymptotic lem 2} For all $q \geq 2$, we have
    $\rho > 1 - \dfrac{1}{2^{q^2-q}}$.
  \item\label{Asymptotic lem 3} Let $q \geq 2$ and let
    $k = (q^2-q)/2$. Then, $\dfrac{1}{\rho^k} - 1 < \dfrac{1}{2^k}$.
  \end{enumerate}
\end{lemma}
\begin{proof}
  \eqref{Asymptotic lem 1} This is a consequence of
  Proposition~\ref{Class sizes prop}\eqref{Class sizes prop 4} and
  Theorem~\ref{theorem:SQD}.

  \eqref{Asymptotic lem 2} The inequality holds when $q=2$.
  For $q \geq 3$, we have
  \begin{equation*}
    (q+1)(2^{q+1}-q-2) < (q+1)\cdot 2^{q+1} \leq 2^{2q},
  \end{equation*}
  and the result follows.

  \eqref{Asymptotic lem 3} By part~\eqref{Asymptotic lem 2},
  $\rho > 1 - \dfrac{1}{4^k}$ for all $q \geq 2$. Thus, by
  Lemma~\ref{1/2^x lem},
  \begin{equation*}
    \dfrac{1}{\rho^k} - 1 < \Big(1 - \dfrac{1}{4^k}\Big)^{-k} -1 < \dfrac{1}{2^k},
  \end{equation*}
  as desired.
\end{proof}

\begin{theorem}\label{Asymptotic thm} Let $k = (q^2-q)/2$. Let $c$ be
  the number of purely core subsets of $\matFq$, and let $c'$ be the
  number of core subsets of $\matFq$ that are not purely core. Then,
  $c' < \Big(\dfrac{1}{2^k}\Big)c$.
\end{theorem}
\begin{proof}
  Let $\mcC_0$ be the union of all the minimal polynomial classes in
  $\matFq$ that are not SQD, and let $\mcC_1$, \ldots, $\mcC_k$ be all
  of the SQD minimal polynomial classes. Assume that
  $S \subseteq \matFq$ is core, but not purely core. Then,
  \begin{equation*}
    S = S_0 \cup S_1 \cup \cdots \cup S_k,
  \end{equation*}
  where $S_i = S \cap \mcC_i$ for each $0 \leq i \leq k$; $S_0$ is
  core; and $S_j$ is noncore for at least one $j$ such that
  $1 \leq j \leq k$. Let $\mcD = \mcC_1 \cup \cdots \cup \mcC_k$ and
  let $T = S_1 \cup \cdots \cup S_k$.

  Let $c_0$ be the number of purely core subsets of $\mcC_0$, and let
  $d$ be the number of purely core subsets of $\mcD$. Then, $c=c_0 d$
  and $d/2^{|\mcD|} = \rho^k$, where $\rho$ is the probability that a
  subset selected uniformly at random from a single SQD class is core. 
  Since $S_0$ is core and $\mcC_0$ consists entirely of non-SQD classes, 
  $S_0$ is purely core by Proposition~\ref{Purely core non-SQD prop}. 
  So, $c_0$ is also equal to the number of possible choices for $S_0$.

  Now, the number of possibilities for $T$ depends on the choice of
  $S_0$. However, by assumption, $T$ is not purely core. Thus, the
  number of possibilities for $T$ is always bounded above by
  $2^{|\mcD|}-d$. So, $c' \leq c_0 (2^{|\mcD|}-d)$. Putting everything
  together, we obtain
\begin{align*}
  \dfrac{c'}{c} \leq \dfrac{c_0(2^{|\mcD|}-d)}{c_0d} = \dfrac{2^{|\mcD|}-d}{d} = \dfrac{1}{\rho^k} - 1.
\end{align*}
The theorem now follows from Lemma~\ref{Asymptotic
  lem}\eqref{Asymptotic lem 3}.
\end{proof}

\begin{theorem}\label{Asymptotic main thm}
  \mbox{}
\begin{enumerate}
\item\label{Asymptotic main thm 1} As $q \to \infty$, almost all subsets of $\matFq$ are purely core. That is, as $q \to \infty$,
\begin{equation*}
\dfrac{\phantom{x}\#\{\text{purely core subsets of $\matFq$}\}\phantom{x}}{\#\{\text{subsets of $\matFq$}\}} \to 1.
\end{equation*}

\item\label{Asymptotic main thm 2} As $q \to \infty$, almost all core subsets of $\matFq$ are purely core. That is, as $q \to \infty$,
\begin{equation*}
\dfrac{\phantom{x}\#\{\text{purely core subsets of $\matFq$}\}\phantom{x}}{\#\{\text{core subsets of $\matFq$}\}} \to 1.
\end{equation*}

\item\label{Asymptotic main thm 3} As $q \to \infty$, almost all subsets of $\matFq$ are core. That is, as $q \to \infty$,
\begin{equation*}
\dfrac{\phantom{x}\#\{\text{core subsets of $\matFq$}\}\phantom{x}}{\#\{\text{subsets of $\matFq$}\}} \to 1.
\end{equation*}
\end{enumerate}

\end{theorem}
\begin{proof}
  \eqref{Asymptotic main thm 1} Let $\mcC_1$, $\mcC_2$, $\mcC_2$, and
  $\mcC_4$ be, respectively, the union of all LIN, IRR, SQR, and SQD
  classes in $\matFq$. For each $1 \leq i \leq 4$, let $c_i$ be the
  number of purely core subsets in $\mcC_i$, and let $c$ be the total
  number of purely core subsets of $\matFq$. Then,
  $c=c_1 c_2 c_3 c_4$. From the counts in Table~\ref{Noncore table},
\begin{align*}
  c_1 &= 2^q,\\
  c_2 &= (2^{q^2-q}-(q^2-q))^{(q^2-q)/2},\\
  c_3 &= (2^{q^2-1}-(q+1)(2^{q+1}-1))^q, \text{ and}\\
  c_4 &= (2^{q^2+q}-(q+1)(2^{q+1}-q-2))^{(q^2-q)/2}.
\end{align*}
Asymptotically,
\begin{equation*}
  c_2 \sim 2^{(q^2-q)(q^2-q)/2}, \quad c_3 \sim 2^{q^3-q}, \quad\text{and} \quad c_4 \sim 2^{(q^2+q)(q^2-q)/2}.
\end{equation*}
Since
$q + \tfrac{1}{2}(q^2-q)^2 + (q^3-q) + \tfrac{1}{2}(q^2+q)(q^2-q) =
q^4$, we see that $c/(2^{q^4}) \to 1$ as $q \to \infty$. This proves
the assertion as $|\matFq| = q^4$.

\eqref{Asymptotic main thm 2} Define $c$ and $c'$ as in
Theorem~\ref{Asymptotic thm}. Then, $c/(c+c') = 1/(1+c'/c)$ and this
tends to 1 as $q \to \infty$ by Theorem~\ref{Asymptotic thm}.

\eqref{Asymptotic main thm 3} This follows from~\eqref{Asymptotic main
  thm 1}.
\end{proof}


\section{Examples of Core Sets that are Not Purely Core}
\label{section:examples}
The matrix ring $\matFq$ always contains subsets that are core, but
not purely core. Theorem~\ref{Asymptotic thm} shows that the
proportion of core sets that are not purely core is small for large $q$.
Nevertheless, it would be useful to enumerate such sets. In
this section, we present some examples to illustrate that---even in a
case where $q$ is small---this can be a difficult
problem.\footnote{Additionally, program code computing the numbers in
  these examples based on the algorithm provided
  in~\cite{Werner:2022:null-ideals} to determine core sets is
  available on the ArXiv page of this paper\arxivid.}

By Proposition~\ref{Purely core non-SQD prop}, any core subset of
$\matFq$ that is not purely core must have a noncore intersection with
at least one SQD class. When $q=2$, there is only one SQD class, so it
is feasible to completely count all of these core sets. As
Example~\ref{M_2(2) example} below shows, even in this simple setting
the calculation involves a rather careful analysis of the situation.

\begin{lemma}\label{L-module lem 3}
  Let $S \subseteq \mcC((x-a)(x-b))$ be a noncore subset of an SQD
  class in~$\matFq$.
\begin{enumerate}
\item\label{L-mod 3, 1}\cite[Lemma 6.3(3)]{Werner:2022:null-ideals} If
  $|S|=1$, then both $\mcL(S,a)$ and $\mcL(S,b)$ are nonzero.
\item\label{L-mod 3, 2}\cite[proof of Lemma 6.3(4)]{Werner:2022:null-ideals} If
  $|S|\geq 2$, then exactly one of $\mcL(S,a)$ or $\mcL(S,b)$ is zero.
\item\label{L-mod 3, 3} If $\mcL(S,a) =\{0\}$, then $x-b$ divides each
  polynomial in $N(S)$. Likewise, if $\mcL(S,b) = \{0\}$, then $x-a$
  divides each polynomial in $N(S)$.
\item\label{L-mod 3, 4} $\mcC((x-a)(x-b))$ contains $(q+1)(2^q-q-1)$
  subsets $T$ such that $\mcL(T,a) = \{0\}$ and $\mcL(T,b) \ne
  \{0\}$. Similarly, there are $(q+1)(2^q-q-1)$ subsets $T$ with
  $\mcL(T,a) \ne \{0\}$ and $\mcL(T,b) = \{0\}$.
\end{enumerate}
\end{lemma}
\begin{proof}
  \eqref{L-mod 3, 3} By~\cite[Proposition
  6.4(2)]{Werner:2022:null-ideals}, any polynomial $f \in N(S)$ can be
  written as
\begin{equation*}
  f(x) = g(x)(x-a)(x-b) + \alpha_1(x-a) + \alpha_2(x-b),
\end{equation*}
where $g(x) \in \matFq[x]$, $\alpha_1 \in \mcL(S,a)$, and
$\alpha_2 \in \mcL(S,b)$. Thus, if $\mcL(S,a) = \{0\}$, then
$\alpha_1=0$ and $x-b$ divides $f$. Likewise, $x-a$ divides $f$ when
$\mcL(S,b) = \{0\}$.

\eqref{L-mod 3, 4} Let $T \subseteq \mcC((x-a)(x-b))$ be such that
$\mcL(T,a) = \{0\}$ but $\mcL(T,b) \ne \{0\}$. Then, $T$ is noncore by
Proposition~\ref{Single class prop}\eqref{Single class prop 3}, and by
part~\eqref{L-mod 3, 1}, $|T| \geq 2$. By Proposition~\ref{Bset
  prop}\eqref{Bset prop 4}, $T \subseteq \mcB(A,b)$ for some
$A \in T$. But, by Lemma~\ref{Bset lem} parts~\eqref{Bset lem 2}
and~\eqref{Bset lem 5}, $|\mcB(A,b)|=q$ and there are $q+1$
possibilities for $\mcB(A,b)$. Since $T$ cannot be the empty set or a
singleton set, there are $(q+1)(2^q - q -1)$ possibilities for~$T$.
\end{proof}

Throughout the remainder of this section, let $\phi_S$ denote the
least common multiple of all polynomials which occur as minimal
polynomial of one of the matrices in $S$.  Recall the following facts
mentioned at the start of Section \ref{section:prelim_and_sizes}.
\begin{fact}\label{fact:Phi}
  Let $S\subseteq \matFq$. The following are equivalent:

  \begin{enumerate}
  \item $S$ is core.
  \item $N(S)$ is generated (as a two-sided ideal) by $\phi_S$.
  \item $N(S)$ does not contain a polynomials of degree less than the
    degree of $\phi_S$.
  \item If $f$ is the minimal polynomial of a matrix in $S$, then $f$
    divides each polynomial in $N(S)$.
  \end{enumerate}
\end{fact}

\begin{example}\label{M_2(2) example}
  We compute the number of core subsets of $\mat{2}$ that are not
  purely core. Since we already know how to enumerate purely core
  subsets, this allows us to give an exact count of the number of core
  subsets of $\mat{2}$. Table~\ref{M_2(2) class table} lists all of
  the minimal polynomial classes in $\mat{2}$, along with the number
  of core and noncore subsets in each class.

\begin{table}[h]
\centering
\begin{tabular}{|c*{3}{|C{2.5cm}}|}
  \hline
  polynomial &  size of class  & \# core subsets in class & \# noncore subsets in class \\
  \hline
  \hline
  $x$ & 1 & 2 & 0\\
  \hline
  $x+1$ & 1 & 2 & 0\\
  \hline
  $x^2+x+1$ & 2 & 2 & 2\\
  \hline
  $x^2$ & 3 & 5 & 3\\
  \hline
  $(x+1)^2$ & 3 & 5 & 3\\
  \hline
  $x(x+1)$ & 6 & 52 & 12\\
  \hline
\end{tabular}
\caption{Minimal polynomial classes in $\mat{2}$}\label{M_2(2) class
  table}
\end{table}

Assume that $S \subseteq \mat{2}$ is core, but is not purely core. Let
$\mcC_0$ be the union of all the non-SQD minimal polynomial classes in
$\mat{2}$. By Proposition~\ref{Purely core non-SQD prop}, we can write
$S = S_0 \cup T$, where $S_0 \subseteq \mcC_0$ is nonempty and core,
and $T \subseteq \mcC(x(x+1))$ is noncore. According to
Fact~\ref{fact:Phi}, $N(S_0)$ is a principal two-sided ideal of
$\mat{2}[x]$. Let $\phi_0\in \F_q[x]$ be the monic generator of
$N(S_0)$. We first argue that either $x \mid \phi_0$ or
$(x+1) \mid \phi_0$. Suppose that neither $x$ nor $x+1$ divides
$\phi_0$. Then, $\phi_0(x)=x^2+x+1$ and
$\phi_S(x) = x (x+1)(x^2+x+1)$. Since $T$ is noncore, there exists a
linear polynomial $f \in N(T)$, and $f\cdot\phi_0$ is an element of
$N(S)$ of degree less than $\deg \phi_S$. This contradicts the fact
that $S$ is core, see Fact~\ref{fact:Phi}.

Thus, at least one of $x$ or $x+1$ divides $\phi_0$. From here, we
will break the problem into three cases, depending on which factors of
$x(x+1)$ divide $\phi_0$.

\vspace{\baselineskip}

\noindent\textbf{Case 1}: $x \mid \phi_0$, but $(x+1) \nmid \phi_0$\\
In this case,
$S_0 \cap \big(\mcC(x+1) \cup \mcC((x+1)^2)\big) =
\varnothing$. Indeed, if this intersection were core but nonempty,
then $x+1$ would divide $\phi_0$ by~Fact~\ref{fact:Phi}; and if the
intersection were noncore, then $S_0$ would be noncore by
Proposition~\ref{Purely core non-SQD prop}. Similar reasoning shows
that $S_0 \cap \big(\mcC(x) \cup \mcC(x^2)\big)$ is core and nonempty,
and $S_0 \cap \mcC(x^2+x+1)$ is core (but may be empty). Thus, the
number of possibilities for $S_0$ is $(2 \cdot 5 -1)\cdot 2=18$.

Note that $\phi_S = \phi_0\cdot(x+1)$. Since $T$ is noncore at least
one of $\mcL(T,0)$ or $\mcL(T,1)$ is nonzero according to
Proposition~\ref{Single class prop}\eqref{Single class prop
  3}. Suppose that $\mcL(T,0) \ne \{0\}$. Then, since $x$ divides
$\phi_0$, we see that $\alpha\phi_0 \in N(S)$ for any nonzero
$\alpha \in \mcL(T,0)$, which is a contradiction to $S$ being core,
cf.~Fact~\ref{fact:Phi}. So, $T$ is such that $\mcL(T,0) = \{0\}$ and
$\mcL(T,1) \ne \{0\}$. By Lemma~\ref{L-module lem 3}\eqref{L-mod 3,
  4}, there are 3 possibilities for $T$. Hence, there are
$54=18\cdot 3$ possibilities for $S$ in Case~1.

\vspace{\baselineskip}

\noindent\textbf{Case 2}: $x \nmid \phi_0$, but $(x+1) \mid \phi_0$\\
Proceeding as in Case~1, we see that there are 54 choices for $S$ in
this case.

\vspace{\baselineskip}

\noindent\textbf{Case 3}: $x \mid \phi_0$ and $(x+1) \mid \phi_0$\\
This time, $S_0 \cap \big(\mcC(x) \cup \mcC(x^2)\big)$ must be core
and nonempty; $S_0 \cap \big(\mcC(x+1) \cup \mcC((x+1)^2)\big)$ must
be core and nonempty; and $S_0 \cap \mcC(x^2+x+1)$ is core (but may be
empty). So, there are $9\cdot 9\cdot 2=162$ possibilities for
$S_0$. Since $x(x+1)$ divides $\phi_0$, each element of $T$ is killed
by $\phi_0$. Thus, $T$ can be any noncore subset of $\mcC(x(x+1))$, of
which there are 12. Hence, in this case there are $1944=162\cdot 12$
choices for $S$.

\vspace{\baselineskip}

Combining these three cases, we see that there are $54+54+1944=2052$
core subsets of $\mat{2}$ that are not purely core. Since there are
$2^3\cdot 5^2\cdot 52=10400$ purely core subsets, we conclude that
$\mat{2}$ contains exactly 12452 core subsets.
\end{example}

In the next two examples, we consider what happens when more than one
SQD class is involved in the construction of core sets that are not
purely core. For situations in which multiple SQD classes are present,
we have found it useful to organize the available information by using
graphs. Given $S \subseteq \matFq$, let $T$ be the collection of all
matrices in $S$ whose minimal polynomial is SQD, and let
$S_0 = S \setminus T$. We know that $S$ is noncore if $S_0$ is noncore
(Proposition~\ref{Purely core non-SQD prop}). If this is not the case,
then form a graph where each vertex corresponds to a root of
$\phi_{T}$, and an edge joins two vertices $a$ and $b$ if and only if
$S \cap \mcC((x-a)(x-b))$ is nonempty. For instance, if $\phi_{T}$ has
roots $a_1$, $a_2$, and $a_3$, then the graph might be
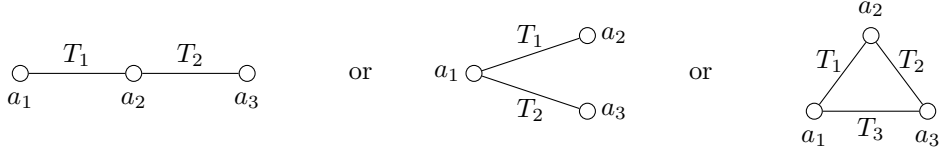
\begin{center}
  \begin{tikzpicture}[scale=1,auto=left,every node/.style={circle,
      draw=black, fill=none, inner sep=0pt, minimum size=6pt}]

    \node[label={[label distance=1pt]270:{$a_1$}}] (a1) at (-1.5,0) {};
    \node[label={[label distance=1pt]270:{$a_2$}}] (b1) at (0,0) {};
    \node[label={[label distance=1pt]270:{$a_3$}}] (c1) at (1.5,0) {};

    \draw (a1) -- node [above,draw=none,fill=none] {$T_1$} (b1);
    \draw (b1) -- node [above,draw=none,fill=none] {$T_2$} (c1);

    \node[draw=none] (OR) at (3,0) {or};

    \node[label={[label distance=1pt]180:{$a_1$}}] (a2) at (4.5,0) {};
    \node[label={[label distance=1pt]0:{$a_2$}}] (b2) at (6,0.5) {};
    \node[label={[label distance=1pt]0:{$a_3$}}] (c2) at (6,-0.5) {};

    \draw (a2) -- node [midway, above, draw=none,fill=none] {$T_1$} (b2);
    \draw (a2) -- node [midway, below, draw=none,fill=none] {$T_2$} (c2);

    \node[draw=none] (OR) at (7.5,0) {or};

    \node[label={[label distance=1pt]270:{$a_1$}}] (a3) at (9,-0.5) {};
    \node[label={[label distance=1pt]90:{$a_2$}}] (b3) at (9.75,0.5) {};
    \node[label={[label distance=1pt]270:{$a_3$}}] (c3) at (10.5,-0.5) {};

    \draw (a3) -- node [above left,draw=none,fill=none] {$T_1$} (b3);
    \draw (b3) -- node [above right,draw=none,fill=none] {$T_2$} (c3);
    \draw (a3) -- node [below,draw=none,fill=none] {$T_3$} (c3);

  \end{tikzpicture}
\end{center}
among other possibilities. Each edge $T_k$ with endpoints $a_i$ and
$a_j$ represents the nonempty set $S \cap \mcC((x-a_i)(x-a_j))$. In
one of these graphs, we will shade the vertex corresponding to the
root $a$ if we know that $x-a$ divides every polynomial in $N(S)$. If
each vertex in the graph is shaded, then $S$ is core by
Fact~\ref{fact:Phi}.

The presence of $\mcL$-modules that equal $\{0\}$ can indicate when to
shade a vertex. Suppose $T_k \subseteq \mcC((x-a_i)(x-a_j))$
corresponds to the edge joining vertices $a_i$ and $a_j$. By
Lemma~\ref{L-module lem 3}\eqref{L-mod 3, 3}, if
$\mcL(T_k,a_i) = \{0\}$, then we shade vertex $a_j$, and if
$\mcL(T_k,a_j) = \{0\}$, then we shade vertex $a_i$. By
Proposition~\ref{Single class prop}\eqref{Single class prop 3}, if
$T_k$ is core, then we shade both vertices.

\begin{example}\label{M_2(3) SQD example}
  Let $\mcC = \mcC(x(x-1)) \cup \mcC((x-1)(x-2))$ in $\mat{3}$. We
  count the number of cores subsets of $\mcC$. Using
  Table~\ref{Noncore table}, we calculate that each of $\mcC(x(x-1))$
  and $\mcC((x-1)(x-2))$ contain $2^{12}- 44= 4052$ core
  subsets. Hence, $\mcC$ contains $4052^2$ purely core subsets.

  Now, let $S \subseteq \mcC$ be core, but not purely core. Then,
  $S = T_1 \cup T_2$, where $T_1 \subseteq \mcC(x(x-1))$,
  $T_2 \subseteq \mcC((x-1)(x-2))$, and at least one of $T_1$ or $T_2$
  is noncore. We will consider three cases, depending on which of
  $T_1$ or $T_2$ (or both) is noncore. Throughout, note that
  $\phi_S(x) = x(x-1)(x-2)$, and $N(S)$ contains no polynomial of
  smaller degree, cf.~Fact~\ref{fact:Phi}.

\vspace{\baselineskip}

\noindent\textbf{Case 1}: $T_1$ is core, but $T_2$ is noncore\\
Since $T_1\cup T_2$ is core, $T_1$ has to be nonempty. The
corresponding graph is
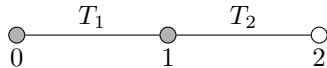
\begin{center}
  \begin{tikzpicture}[scale=1,auto=left,every node/.style={circle,
      draw=black, fill=none, inner sep=0pt, minimum size=6pt}]

    \node[fill = black!30, label={[label distance=1pt]270:{0}}] (0) at (-2,0) {};
    \node[fill = black!30, label={[label distance=1pt]270:{1}}] (1) at (0,0) {};
    \node[label={[label distance=1pt]270:{2}}] (2) at (2,0) {};

    \draw (0) -- node [above,draw=none,fill=none] {$T_1$} (1);
    \draw (1) -- node [above,draw=none,fill=none] {$T_2$} (2);
\end{tikzpicture}
\end{center}
where both vertex 0 and 1 are shaded because $T_1$ is assumed to be
core. We can shade vertex $2$ if $\mcL(T_2,1) = \{0\}$. However, if
$\mcL(T_2,1) \ne \{0\}$, then $\alpha x(x-1) \in N(S)$ for each
nonzero $\alpha \in \mcL(T_2,1)$. Thus, $S$ is core if and only if
$\mcL(T_2,1) = \{0\}$. By Lemma~\ref{L-module lem 3}\eqref{L-mod 3,
  4}, there are $16$ possibilities for $T_2$. Thus, there are
$4051\cdot 16$ possibilities for $S$ in this case.

\vspace{\baselineskip}

\noindent\textbf{Case 2}: $T_1$ is noncore, but $T_2$ is core\\
This time, the corresponding graph is
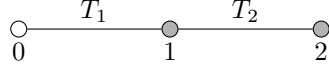
\begin{center}
  \begin{tikzpicture}[scale=1,auto=left,every node/.style={circle,
      draw=black, fill=none, inner sep=0pt, minimum size=6pt}]

    \node[label={[label distance=1pt]270:{0}}] (0) at (-2,0) {};
    \node[fill = black!30, label={[label distance=1pt]270:{1}}] (1) at (0,0) {};
    \node[fill = black!30, label={[label distance=1pt]270:{2}}] (2) at (2,0) {};

    \draw (0) -- node [above,draw=none,fill=none] {$T_1$} (1);
    \draw (1) -- node [above,draw=none,fill=none] {$T_2$} (2);
    \end{tikzpicture}
\end{center}
Proceeding as in Case 1, we find that there are $4051\cdot 16$
possibilities for $S$.

\vspace{\baselineskip}

\noindent\textbf{Case 3}: both $T_1$ and $T_2$ are noncore\\
Here, the graph is
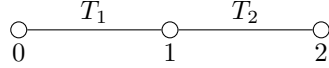
\begin{center}
  \begin{tikzpicture}[scale=1,auto=left,every node/.style={circle,
      draw=black, fill=none, inner sep=0pt, minimum size=6pt}]

    \node[label={[label distance=1pt]270:{0}}] (0) at (-2,0) {};
    \node[label={[label distance=1pt]270:{1}}] (1) at (0,0) {};
    \node[label={[label distance=1pt]270:{2}}] (2) at (2,0) {};

    \draw (0) -- node [above,draw=none,fill=none] {$T_1$} (1);
    \draw (1) -- node [above,draw=none,fill=none] {$T_2$} (2);
\end{tikzpicture}
\end{center}
With the argument used in Case~1, in order to shade both vertices 0
and 2, we need $\mcL(T_1,1) = \{0\}$ and $\mcL(T_2,1) = \{0\}$, which
produces the graph
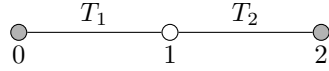
\begin{center}
  \begin{tikzpicture}[scale=1,auto=left,every node/.style={circle,
      draw=black, fill=none, inner sep=0pt, minimum size=6pt}]

    \node[fill = black!30, label={[label distance=1pt]270:{0}}] (0) at (-2,0) {};
    \node[label={[label distance=1pt]270:{1}}] (1) at (0,0) {};
    \node[fill = black!30, label={[label distance=1pt]270:{2}}] (2) at (2,0) {};

    \draw (0) -- node [above,draw=none,fill=none] {$T_1$} (1);
    \draw (1) -- node [above,draw=none,fill=none] {$T_2$} (2);
\end{tikzpicture}
\end{center}
Now, since both $T_1$ and $T_2$ are noncore, both $\mcL(T_1,0)$ and
$\mcL(T_2,2)$ are nonzero (Proposition~\ref{Single class
  prop}\eqref{Single class prop 3}). In this situation,~\cite[Theorem
6.6]{Werner:2022:null-ideals} shows that $S$ is core if and only if
$\mcL(T_1,0) \ne \mcL(T_2,2)$. By Lemma~\ref{L-module lem
  2}\eqref{L-mod 2, 1}, there are 4 choices for each $\mcL$-module,
and hence $4^2-4=12$ ways to choose $\mcL(T_1,0)$ and $\mcL(T_2,2)$ so
that they are not equal. Once $\mcL(T_1,0)$ has been specified, there
are $2^3-(3+1) = 4$ choices for $T_1$, because $T_1$ must be a subset
of some $\mcB(-,0)$ of size at least 2 (Lemma~\ref{L-module lem
  3}\eqref{L-mod 3, 1}). Likewise, there are 4 choices for $T_2$. In
summary, there are $12\cdot 4^2$ possibilities for $S$ in Case 3.

\vspace{\baselineskip}

We conclude that $\mcC$ contains $4052^2$ purely core and
$2\cdot 4051\cdot 16 + 12\cdot 4^2$ core subsets that are not purely
core.
\end{example}

In our last example, we examine the effect of mixing an SQR class with
two SQD classes.

\begin{example}\label{M_2(3) example}
  Let $\mcC = \mcC(x^2) \cup \mcC(x(x-1)) \cup \mcC((x-1)(x-2))$ in
  $\mat{3}$. We count the number of core subsets of $\mcC$. Via
  Table~\ref{Noncore table}, $\mcC(x^2)$ contains $2^8-12=244$ core
  subsets, and each SQD class contains $4052$ core subsets
  (cf.~Example~\ref{M_2(3) SQD example}). So, $\mcC$ contains
  $244\cdot 4052^2$ purely core subsets.

  From here, let $S \subseteq \mcC$ be core, but not purely
  core. Then, $S = S_0 \cup T_1 \cup T_2$, where
  $S_0 \subseteq \mcC(x^2)$ is core, $T_1 \subseteq \mcC(x(x-1))$, and
  $T_2 \subseteq \mcC((x-1)(x-2))$ (Proposition~\ref{Single class
    prop}\eqref{Single class prop 3}). If $S_0 = \varnothing$, then
  $T_1 \cup T_2$ is core, but not purely core. From
  Example~\ref{M_2(3) SQD example}, we know that there are
  $2\cdot 4051\cdot 16 + 12\cdot 4^2$ choices for $S$ when
  $S_0 = \varnothing$.

  For the remainder of the example, we will assume that $S_0$ is
  nonempty, which gives 243 possibilities for $S_0$. Moreover, at
  least one of $T_1$ or $T_2$ must be noncore. We have
  $\phi_S(x)=x^2(x-1)(x-2)$, and since $S_0$ is core, we know that
  $x^2$ divides each polynomial in $N(S)$
  (cf.~Fact~\ref{fact:Phi}). As in Example~\ref{M_2(3) SQD example},
  we will consider cases depending on $T_1$ and $T_2$, and will draw
  graphs corresponding to each case. Because $x^2$ divides every
  polynomial in $N(S)$, vertex 0 will be shaded in each graph.

\vspace{\baselineskip}

\noindent\textbf{Case 1}: $T_1$ is core, but $T_2$ is noncore\\
We must have $T_1 \ne \varnothing$, since otherwise
$x^2\cdot f \in N(S)$ for any $f \in N(T_2)$ of degree~1, and
$x^2\cdot f$ has degree less than $\phi_S$
(cf.~Fact~\ref{fact:Phi}). For this case, the graph is
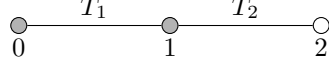
\begin{center}
  \begin{tikzpicture}[scale=1,auto=left,every node/.style={circle,
      draw=black, fill=none, inner sep=0pt, minimum size=6pt}]

    \node[fill = black!30, label={[label distance=1pt]270:{0}}] (0) at (-2,0) {};
    \node[fill = black!30, label={[label distance=1pt]270:{1}}] (1) at (0,0) {};
    \node[label={[label distance=1pt]270:{2}}] (2) at (2,0) {};

    \draw (0) -- node [above,draw=none,fill=none] {$T_1$} (1);
    \draw (1) -- node [above,draw=none,fill=none] {$T_2$} (2);
  \end{tikzpicture}
\end{center}
Analogous to the argument in Example~\ref{M_2(3) SQD example} Case~1,
in order to shade vertex 2, we need $\mcL(T_2,1) = \{0\}$, and then
$\mcL(T_2,2) \ne \{0\}$ because $T_2$ is noncore
(Proposition~\ref{Single class prop}\eqref{Single class prop 3}). From
Lemma~\ref{L-module lem 3}\eqref{L-mod 3, 4}, there are $16$
possibilities for~$T_2$. Hence, there are $243\cdot 4051\cdot 16$
choices for $S$.

\vspace{\baselineskip}

\noindent\textbf{Case 2}: $T_1$ is noncore, but $T_2$ is core\\
If $T_2 = \varnothing$, then the graph is
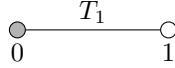
\begin{center}
  \begin{tikzpicture}[scale=1,auto=left,every node/.style={circle,
      draw=black, fill=none, inner sep=0pt, minimum size=6pt}]
    \node[fill = black!30, label={[label distance=1pt]270:{0}}] (0) at (-2,0) {};
    \node[label={[label distance=1pt]270:{1}}] (1) at (0,0) {};

    \draw (0) -- node [above,draw=none,fill=none] {$T_1$} (1);
\end{tikzpicture}
\end{center}
Again, to shade vertex 1, we need $\mcL(T_1,0) = \{0\}$. Using
Proposition~\ref{Single class prop}\eqref{Single class prop 3} and
Lemma~\ref{L-module lem 3}\eqref{L-mod 3, 4}, when $T_2 = \varnothing$
there are 16 choices for $T_1$.

Assume now that $T_2 \ne \varnothing$. Graphically, we have
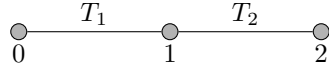
\begin{center}
  \begin{tikzpicture}[scale=1,auto=left,every node/.style={circle,
      draw=black, fill=none, inner sep=0pt, minimum size=6pt}]

    \node[fill = black!30, label={[label distance=1pt]270:{0}}] (0) at (-2,0) {};
    \node[fill = black!30, label={[label distance=1pt]270:{1}}] (1) at (0,0) {};
    \node[fill = black!30, label={[label distance=1pt]270:{2}}] (2) at (2,0) {};

    \draw (0) -- node [above,draw=none,fill=none] {$T_1$} (1);
    \draw (1) -- node [above,draw=none,fill=none] {$T_2$} (2);
\end{tikzpicture}
\end{center}
This means that $T_1$ could be any noncore subset of $\mcC(x(x-1))$,
and $S$ will still be core. Thus, if $T_2 \ne \varnothing$, then there
are $44$ choices for $T_1$ and $4051$ choices for $T_2$ (see
Table~\ref{Noncore table}). In total, Case 2 produces
$243\cdot 16+243\cdot 44\cdot 4051$ possibilities for~$S$.

\vspace{\baselineskip}

\noindent\textbf{Case 3}: both $T_1$ and $T_2$ are noncore\\
For this case, the initial graph is
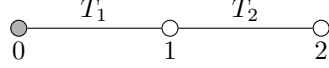
\begin{center}
  \begin{tikzpicture}[scale=1,auto=left,every node/.style={circle,
      draw=black, fill=none, inner sep=0pt, minimum size=6pt}]

    \node[fill = black!30, label={[label distance=1pt]270:{0}}] (0) at (-2,0) {};
    \node[label={[label distance=1pt]270:{1}}] (1) at (0,0) {};
    \node[label={[label distance=1pt]270:{2}}] (2) at (2,0) {};

    \draw (0) -- node [above,draw=none,fill=none] {$T_1$} (1);
    \draw (1) -- node [above,draw=none,fill=none] {$T_2$} (2);
  \end{tikzpicture}
\end{center}
In order to shade vertex 2, we need $\mcL(T_2,1) = \{0\}$, which
forces $\mcL(T_2,2) \ne \{0\}$ by Proposition~\ref{Single class
  prop}\eqref{Single class prop 3}. The graph becomes
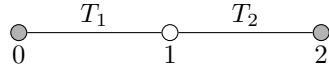
\begin{center}
  \begin{tikzpicture}[scale=1,auto=left,every node/.style={circle,
      draw=black, fill=none, inner sep=0pt, minimum size=6pt}]

    \node[fill = black!30, label={[label distance=1pt]270:{0}}] (0) at (-2,0) {};
    \node[label={[label distance=1pt]270:{1}}] (1) at (0,0) {};
    \node[fill = black!30, label={[label distance=1pt]270:{2}}] (2) at (2,0) {};

    \draw (0) -- node [above,draw=none,fill=none] {$T_1$} (1);
    \draw (1) -- node [above,draw=none,fill=none] {$T_2$} (2);
  \end{tikzpicture}
\end{center}
There are now two possible ways that vertex 1 could be
shaded. By~\cite[Theorem 6.6]{Werner:2022:null-ideals}, this will
occur if and only $\mcL(T_1,0) \cap \mcL(T_2,2)$ is zero.  It could be
that $\mcL(T_1,0) = \{0\}$, which leads to $\mcL(T_1,1) \ne \{0\}$ by
Proposition~\ref{Single class prop}\eqref{Single class prop 3},
yielding 16 choices each for $T_1$ and $T_2$ by Lemma~\ref{L-module
  lem 3}\eqref{L-mod 3, 4}, and $243 \cdot 16^2$ choices for
$S$. Otherwise, $\mcL(T_1,0)$ is nonzero but not equal to
$\mcL(T_2,2)$. This situation is similar to Case 3 of
Example~\ref{M_2(3) SQD example}. There are 12 ways to select the
$\mcL$-modules so that they are nonzero and not equal. Once
$\mcL(T_2,2)$ is fixed, there are $2^3-(3+1)=4$ choices for $T_2$,
since $T_2$ is contained in some $\mcB(-,2)$ and $|T_2| \geq 2$. The
set $T_1$ is likewise contained in some $\mcB(-,0)$, but $T_1$ could
be a singleton set. Hence, there are $2^3-1=7$ possibilities for
$T_1$. In total, Case 3 leads to
$243 \cdot 16^2 + 243 \cdot 12 \cdot 7 \cdot 4$ choices for $S$.

\vspace{\baselineskip}

Combining all of our work in Example~\ref{M_2(3) example}, we see that
$\mcC$ contains $244 \cdot 4052^2$ purely core and
\begin{equation*}
  243\cdot(4051\cdot 16 + 16 + 44\cdot 4051 + 16^2 + 12\cdot 7 \cdot 4)
\end{equation*}
core but not purely core subsets.
\end{example}

As $q$ grows, the number of available SQD classes also grows, and the
potential interactions among SQD classes and non-SQD classes become
increasingly complicated. The examples above demonstrate that
enumerating core sets that are not purely core is feasible in
particular situations, but producing an exact formula for the number
of such sets is difficult. Nevertheless, we feel that the techniques
and perspectives used in these examples could be applicable for the
study of core sets and null ideals over $M_n(\F_q)$ with $n \geq 3$,
or over other rings or algebras.


\bibliographystyle{amsplainurl}
\bibliography{bibliography}

\providecommand{\bysame}{\leavevmode\hbox to3em{\hrulefill}\thinspace}
\providecommand{\MR}{\relax\ifhmode\unskip\space\fi MR }
\providecommand{\MRhref}[2]{%
  \href{http://www.ams.org/mathscinet-getitem?mr=#1}{#2}
}
\providecommand{\href}[2]{#2}
\begin{thebibliography}{10}

\bibitem{Brown:1998:null-ideals-spanning-rank}
William~C. Brown, \href{http://dx.doi.org/10.1080/00927879808826285}{\emph{Null
  ideals and spanning ranks of matrices}}, Comm. Algebra \textbf{26} (1998),
  no.~8, 2401--2417. \MR{1627925}

\bibitem{Brown:1999:null-ideals-spanning-rank}
\bysame, \href{http://dx.doi.org/10.1080/00927879908826807}{\emph{Null ideals
  and spanning ranks of matrices. {II}}}, Comm. Algebra \textbf{27} (1999),
  no.~12, 6051--6067. \MR{1726293}

\bibitem{Brown:2005:null-ideals}
\bysame, \href{http://dx.doi.org/10.1080/00927870500274820}{\emph{Null ideals
  of matrices}}, Comm. Algebra \textbf{33} (2005), no.~12, 4491--4504.
  \MR{2188324}

\bibitem{Cahen-Chabert:1997:ivp}
Paul-Jean Cahen and Jean-Luc Chabert,
  \href{http://dx.doi.org/10.1090/surv/048}{\emph{Integer-valued polynomials}},
  Mathematical Surveys and Monographs, vol.~48, American Mathematical Society,
  Providence, RI, 1997. \MR{1421321}

\bibitem{Cahen-Chabert:2016:ivp-survey}
\bysame,
  \href{http://dx.doi.org/10.4169/amer.math.monthly.123.4.311}{\emph{What you
  should know about integer-valued polynomials}}, Amer. Math. Monthly
  \textbf{123} (2016), no.~4, 311--337. \MR{3493376}

\bibitem{Chooi-Kwa-Lim:2017}
Wai~Leong Chooi, Kiam~Heong Kwa, and Ming-Huat Lim,
  \href{http://dx.doi.org/10.1016/j.laa.2016.11.023}{\emph{Coherence invariant
  maps on tensor products}}, Linear Algebra Appl. \textbf{516} (2017), 24--46.
  \MR{3589703}

\bibitem{DummitFoote}
David~S. Dummit and Richard~M. Foote, \emph{Abstract algebra}, third ed., John
  Wiley \& Sons, Inc., Hoboken, NJ, 2004. \MR{2286236}

\bibitem{Evrard-Fares-Johnson:2013:ivp-lower-triangular-matrices}
S.~Evrard, Y.~Fares, and K.~Johnson,
  \href{http://dx.doi.org/10.1007/s00605-013-0481-6}{\emph{Integer valued
  polynomials on lower triangular integer matrices}}, Monatsh. Math.
  \textbf{170} (2013), no.~2, 147--160. \MR{3041736}

\bibitem{Evrard-Johnson:2015:ivp-2by2}
S.~Evrard and K.~Johnson,
  \href{http://dx.doi.org/10.1016/j.jalgebra.2015.06.011}{\emph{The ring of
  integer valued polynomials on {$2\times 2$} matrices and its integral
  closure}}, J. Algebra \textbf{441} (2015), 660--677. \MR{3391941}

\bibitem{Frisch:2017:ivp-upper-triangular}
Sophie Frisch,
  \href{http://dx.doi.org/10.1007/s00605-016-1013-y}{\emph{Polynomial functions
  on upper triangular matrix algebras}}, Monatsh. Math. \textbf{184} (2017),
  no.~2, 201--215. \MR{3696109}

\bibitem{Godsil-Royle:2001}
Chris Godsil and Gordon Royle,
  \href{http://dx.doi.org/10.1007/978-1-4613-0163-9}{\emph{Algebraic graph
  theory}}, Graduate Texts in Mathematics, vol. 207, Springer-Verlag, New York,
  2001. \MR{1829620}

\bibitem{Heuberger-Rissner:2017:null-ideals}
Clemens Heuberger and Roswitha Rissner,
  \href{http://dx.doi.org/10.1016/j.laa.2017.03.028}{\emph{Computing
  {$J$}-ideals of a matrix over a principal ideal domain}}, Linear Algebra
  Appl. \textbf{527} (2017), 12--31. \MR{3647467}

\bibitem{HornJohnson}
Roger~A. Horn and Charles~R. Johnson, \emph{Topics in matrix analysis},
  Cambridge University Press, Cambridge, 1991.

\bibitem{Huang-Huang-Li-Sze:2014}
Li-Ping Huang, Zejun Huang, Chi-Kwong Li, and Nung-Sing Sze,
  \href{http://dx.doi.org/10.1016/j.laa.2013.12.030}{\emph{Graphs associated
  with matrices over finite fields and their endomorphisms}}, Linear Algebra
  Appl. \textbf{447} (2014), 2--25. \MR{3200203}

\bibitem{Huang-Semrl:2016}
Wen-ling Huang and Peter \u{S}emrl,
  \href{http://dx.doi.org/10.1016/j.laa.2014.08.014}{\emph{The optimal version
  of {H}ua's fundamental theorem of geometry of square matrices---the low
  dimensional case}}, Linear Algebra Appl. \textbf{498} (2016), 21--57.
  \MR{3478548}

\bibitem{Hyun-Neiger-Schost:2021:lin-rec-seq}
Seung~Gyu Hyun, Vincent Neiger, and \'{E}ric Schost,
  \href{http://dx.doi.org/10.1145/3452143.3465533}{\emph{Algorithms for
  linearly recurrent sequences of truncated polynomials}}, I{SSAC}
  '21---{P}roceedings of the 2021 {I}nternational {S}ymposium on {S}ymbolic and
  {A}lgebraic {C}omputation, ACM, New York, [2021] \copyright 2021,
  pp.~201--208. \MR{4398785}

\bibitem{Lam:2001:non-comm-rings}
T.~Y. Lam, \href{http://dx.doi.org/10.1007/978-1-4419-8616-0}{\emph{A first
  course in noncommutative rings}}, second ed., Graduate Texts in Mathematics,
  vol. 131, Springer-Verlag, New York, 2001. \MR{1838439}

\bibitem{Naghipour-Rismanchian-Hafshejani:2017:ivp-matrix-rings}
A.~R. Naghipour, M.~R. Rismanchian, and J.~Sedighi~Hafshejani,
  \href{http://dx.doi.org/10.1080/00927872.2016.1222407}{\emph{Some results on
  the integer-valued polynomials over matrix rings}}, Comm. Algebra \textbf{45}
  (2017), no.~4, 1675--1686. \MR{3576685}

\bibitem{Orel:2012}
Marko Orel, \href{http://dx.doi.org/10.1007/s10801-011-0318-0}{\emph{Adjacency
  preservers, symmetric matrices, and cores}}, J. Algebraic Combin. \textbf{35}
  (2012), no.~4, 633--647. \MR{2902704}

\bibitem{Orel:2016}
\bysame, \href{http://dx.doi.org/10.1016/j.laa.2014.10.034}{\emph{Adjacency
  preservers on invertible hermitian matrices {I}}}, Linear Algebra Appl.
  \textbf{499} (2016), 99--128. \MR{3478888}

\bibitem{Rissner:2016:null-ideals}
Roswitha Rissner, \href{http://dx.doi.org/10.1016/j.laa.2016.01.004}{\emph{Null
  ideals of matrices over residue class rings of principal ideal domains}},
  Linear Algebra Appl. \textbf{494} (2016), 44--69. \MR{3455685}

\bibitem{Hafshejani-Naghipour-Rismanchian:2020:ivp-block-matrices}
J.~Sedighi~Hafshejani, A.~R. Naghipour, and M.~R. Rismanchian,
  \href{http://dx.doi.org/10.1142/S021949882050053X}{\emph{Integer-valued
  polynomials over block matrix algebras}}, J. Algebra Appl. \textbf{19}
  (2020), no.~3, 2050053, 17. \MR{4082437}

\bibitem{Hafshejani-Naghipour-Sakzad:2019:ivp-matrix-rings}
J.~Sedighi~Hafshejani, A.~R. Naghipour, and A.~Sakzad,
  \href{http://dx.doi.org/10.1080/00927872.2018.1499926}{\emph{Integer-valued
  polynomials over subsets of matrix rings}}, Comm. Algebra \textbf{47} (2019),
  no.~3, 1077--1090. \MR{3938540}

\bibitem{Swartz-Werner:2023:null-ideal-3by3}
Eric Swartz and Nicholas~J. Werner,
  \href{http://dx.doi.org/10.1080/03081087.2023.2271639}{\emph{Null ideals of
  sets of {$3\times3$} similar matrices with irreducible characteristic
  polynomial}}, Linear Multilinear Algebra \textbf{72} (2024), no.~15,
  2516--2538. \MR{4803060}

\bibitem{Werner:2012:ivp-matrix-rings}
Nicholas~J. Werner,
  \href{http://dx.doi.org/10.1080/00927872.2011.606859}{\emph{Integer-valued
  polynomials over matrix rings}}, Comm. Algebra \textbf{40} (2012), no.~12,
  4717--4726. \MR{2989677}

\bibitem{Werner:2017:ivp-survey}
\bysame,
  \href{http://dx.doi.org/10.1007/978-3-319-65874-2_18}{\emph{Integer-valued
  polynomials on algebras: A survey of recent results and open questions}},
  Rings, Polynomials, and Modules (Marco Fontana, Sophie Frisch, Sarah Glaz,
  Francesca Tartarone, and Paolo Zanardo, eds.), Springer International
  Publishing, Cham, 2017, pp.~353--375.

\bibitem{Werner:2022:null-ideals}
\bysame, \href{http://dx.doi.org/10.1016/j.laa.2022.02.007}{\emph{Null ideals
  of subsets of matrix rings over fields}}, Linear Algebra Appl. \textbf{642}
  (2022), 50--72. \MR{4381683}

\end{thebibliography}

\end{document}